\documentclass[12pt]{amsart}
\usepackage{a4wide}
\usepackage[T1]{fontenc}
\usepackage{amssymb,amsmath,amsthm,latexsym}
\usepackage{mathrsfs}
\usepackage[usenames,dvipsnames]{color}
\usepackage{euscript}
\usepackage{graphicx}
\usepackage{mdwlist}
\usepackage{enumerate}
\usepackage{enumitem}
\usepackage{xcolor}
\usepackage{mathtools,dsfont,wasysym}
\usepackage{stmaryrd}
\usepackage{ulem}
\usepackage{hyperref}
\usepackage{cancel}
\usepackage{float}

\hypersetup{colorlinks=true, linkcolor=blue, citecolor=red, urlcolor=cyan}

\newtheorem{theorem}{Theorem}[section]
\newtheorem{lemma}[theorem]{Lemma}
\newtheorem{corollary}[theorem]{Corollary}

\newtheorem{proposition}[theorem]{Proposition}

\newcounter{maintheorem}

\theoremstyle{remark}
\newtheorem{remark}[theorem]{Remark}

\theoremstyle{definition}

\numberwithin{equation}{section}
\makeatother

\makeatletter
\renewcommand{\tocsection}[3]{%
\indentlabel{\@ifnotempty{#2}{\bfseries\ignorespaces#1 #2\quad}}\bfseries#3}
\renewcommand{\tocsubsection}[3]{%
\indentlabel{\@ifnotempty{#2}{\ignorespaces#1 #2\quad}}#3}

\newcommand\@dotsep{4.5}
\def\@tocline#1#2#3#4#5#6#7{\relax
\ifnum #1>\c@tocdepth 
\else
\par \addpenalty\@secpenalty\addvspace{#2}%
\begingroup \hyphenpenalty\@M
\@ifempty{#4}{%
\@tempdima\csname r@tocindent\number#1\endcsname\relax
}{%
\@tempdima#4\relax
}%
\parindent\z@ \leftskip#3\relax \advance\leftskip\@tempdima\relax
\rightskip\@pnumwidth plus1em \parfillskip-\@pnumwidth
#5\leavevmode\hskip-\@tempdima{#6}\nobreak
\leaders\hbox{$\m@th\mkern \@dotsep mu\hbox{.}\mkern \@dotsep mu$}\hfill
\nobreak
\hbox to\@pnumwidth{\@tocpagenum{\ifnum#1=1\bfseries\fi#7}}\par
\nobreak
\endgroup
\fi}
\AtBeginDocument{%
\expandafter\renewcommand\csname r@tocindent0\endcsname{0pt}
}
\def\l@subsection{\@tocline{2}{0pt}{2.5pc}{5pc}{}}
\makeatother

\newcommand{\sna}{\operatorname{SNA}}

\newcommand{\bbr}{\mathbb{R}}
\newcommand{\bbn}{\mathbb{N}}

\newcommand{\N}{\mathbb{N}}

\newcommand{\calf}{\mathcal{F}}

\newcommand{\fm}{\calf(M)}
\newcommand{\na}{\operatorname{NA}}

\newcommand{\pna}{\operatorname{PNA}}
\newcommand{\PNA}{\operatorname{PNA}}

\newcommand{\ldira}{\operatorname{LDirA}}

\newcommand{\lip}{\operatorname{Lip}_0}

\newcommand{\Lip}{\operatorname{Lip}_0}

\newcommand{\der}{\operatorname{Der}}

\newcommand{\sign}{\operatorname{sign}}

\newcommand{\eps}{\varepsilon}

\DeclareMathOperator{\supp}{supp}

\DeclareMathOperator{\spann}{span}

\DeclareFontFamily{U}{mathx}{}
\DeclareFontShape{U}{mathx}{m}{n}{<-> mathx10}{}
\DeclareSymbolFont{mathx}{U}{mathx}{m}{n}
\DeclareMathAccent{\widehat}{0}{mathx}{"70}
\DeclareMathAccent{\widecheck}{0}{mathx}{"71}

\DeclareMathOperator{\NA}{NA}

\newcommand{\neib}{\operatorname{Neib}}

\renewcommand{\epsilon}{\varepsilon}

\renewcommand{\subset}{\subseteq}

\newlength\Colsep
\begin{document}
\title[Linear structures of norm-attaining Lipschitz functions]{Linear structures of norm-attaining Lipschitz functions and their complements}

\author[Choi]{Geunsu Choi}
\address[Choi]{Department of Mathematics Education, Sunchon National University, 57922 Jeonnam, Republic of Korea \newline
\href{http://orcid.org/0000-0002-4321-1524}{ORCID: \texttt{0000-0002-4321-1524}}}
\email{\texttt{gschoi@scnu.ac.kr}}

\author[Jung]{Mingu Jung} 
\address[Jung]{June E Huh Center for Mathematical Challenges, Korea Institute for Advanced Study, 02455 Seoul, Republic of Korea\newline
\href{https://orcid.org/0000-0003-2240-2855}{ORCID: \texttt{0000-0003-2240-2855}}}
\email{jmingoo@kias.re.kr}

\author[Lee]{Han Ju Lee}
\address[Lee]{Department of Mathematics Education, Dongguk University, 04620 Seoul, Republic of Korea \newline
\href{https://orcid.org/0000-0001-9523-2987}{ORCID: \texttt{0000-0001-9523-2987}}}
\email{\texttt{hanjulee@dgu.ac.kr}}

\author[Rold\'an]{\'Oscar Rold\'an}
\address[Rold\'an]{Department of Mathematics Education, Dongguk University, 04620 Seoul, Republic of Korea \newline
\href{https://orcid.org/0000-0002-1966-1330}{ORCID: \texttt{0000-0002-1966-1330}}}
\email{\texttt{oscar.roldan@uv.es}}


\date{\today}                                           


\begin{abstract}
We solve two main questions on linear structures of \textup{(}non-\textup{)}norm-attaining Lipschitz functions. First, we show that for every infinite metric space $M$, the set consisting of Lipschitz functions on $M$ which do not strongly attain their norm and the zero contains an isometric copy of $\ell_\infty$, and moreover, those functions can be chosen not to attain their norm as functionals on the Lipschitz-free space over $M$. Second, we prove that for every infinite metric space $M$, neither the set of strongly norm-attaining Lipschitz functions on $M$ nor the union of its complement with zero is ever a linear space. Furthermore, we observe that the set consisting of Lipschitz functions which cannot be approximated by strongly norm-attaining ones and the zero element contains $\ell_\infty$ isometrically in all the known cases. Some natural observations and spaceability results are also investigated for Lipschitz functions that attain their norm in one way but do not in another, for several norm-attainment notions considered in the literature.
\end{abstract}

\maketitle

\hypersetup{linkcolor=black}

\makeatletter \def\l@subsection{\@tocline{2}{0pt}{1pc}{5pc}{}} \def\l@subsection{\@tocline{2}{0pt}{3pc}{6pc}{}} \makeatother


\hypersetup{linkcolor=blue}

\section{Introduction}

A subset $A$ of a vector space is said to be \textit{lineable} if $A\cup\{0\}$ contains an infinite-dimensional linear space, and \textit{spaceable} if $A\cup\{0\}$ contains a \textit{closed} infinite-dimensional linear space. 
The study of lineability dates back to a result by B. Levine and D. Milman \cite{LM} in 1940, which states that the subset of all functions of bounded variation in $C[0,1]$ does not contain a closed infinite dimensional subspace. In this direction, V.I. Gurariy \cite{Gurariy66} first proved that the set of nowhere differentiable functions on $[0,1]$ is lineable. Afterwards, it was shown by V.P. Fonf, V.I. Gurariy, and M.I. Kadec \cite{FGK99} that such a set is also spaceable. To this day, 
plenty of research has been done on large linear structures of functions that exhibit pathological behavior or certain properties. For a complete background on lineability and spaceability, we refer the reader to the monograph \cite{ABPS} and the seminal papers \cite{AGS05, EGS14, GQ04}.

Throughout the paper, we respectively denote by $X^*$, $B_X$, and $S_X$ the dual space, the unit ball, and the unit sphere of a Banach space $X$ over the \textit{real} field $\mathbb{R}$. The symbol $\mathcal{L}(X,Y)$ stands for the Banach space of all bounded linear operators from $X$ to a Banach space $Y$. Recall that an operator $T \in \mathcal{L}(X,Y)$ {\it attains its norm}, or it is {\it norm-attaining}, if there exists $x_0 \in S_X$ such that $\|T\| = \|T(x_0)\|$. We denote by $\NA(X, Y)$ the subset of $\mathcal{L}(X, Y)$ of all bounded linear operators which attain their norm. When we deal with the scalar case, in order to simplify the notations, we use only the symbol $\NA(X)$ instead of $\NA(X,\bbr)$.

 In general, the set $\NA(X, Y)$  is not always a subspace of $\mathcal L(X,Y)$. For instance, $\NA(\ell_1)$ is nothing but the set $\{x \in \ell_{\infty}: \|x\|_{\infty} = \max_{n \in \N} |x_n| \}$, which is not a subspace of $\ell_\infty$. However, it is clear that if $X = Z^*$ is an infinite-dimensional dual Banach space, then $\na(X)$ is spaceable as $Z$ is contained in $\na(X)$ by the Hahn-Banach theorem, and if $X$ is reflexive, $\na(X)$ is a Banach space (isometrically isomorphic to $X^*$) by the result of R.C. James (see \cite{James57, James64}). Moreover, given a separable Banach space $X$, if $W$ is a closed separating subspace of $X^*$ such that $W \subseteq \NA(X)$, the space $W$ turns out to be an isometric predual of $X$ \cite{PP}. In 2001, G. Godefroy asked whether given a Banach space $X$ of dimension at least $2$, the set $\na(X)$ contains a $2$-dimensional linear space \cite[Problem III]{Godefroy01}, and this question was studied deeply in \cite{BG}. Finally, M. Rmoutil \cite{Rmoutil17} solved this question in the negative by proving that there is a renorming $X$ of $c_0$ due to C. Read (see \cite{Read18}) for which $\na(X)$ does not contain a $2$-dimensional linear space. For more results on lineability and spaceability in the context of norm-attaining operators, we refer to \cite{AAAG07, BG, FGMR, GPacheco08, GP2010, GP2015, PT}. 

  Similarly, for non-reflexive $X$, the set $X^*  \setminus \NA(X)$ consisting of all non-norm-attaining functionals on $X$ is also often spaceable.
  For instance, under certain natural conditions the sets $\mathcal{C}(K)^* \setminus \NA(\mathcal{C}(K))$ and $L_1(\mu)^* \setminus \NA(L_1(\mu))$ are spaceable \cite[Theorem 2.5, Theorem 2.7]{AAAG07}. However, note that there are non-reflexive Banach spaces such that $(X^*  \setminus \NA(X)) \cup\{0\}$ does not even contain $2$-dimensional linear spaces (see \cite[Example 3.1]{GPacheco08}). In the context of linear operators, D. Pellegrino and E. Teixeira \cite{PT} showed that, whenever $1 \leq p < \infty$ and $Y$ contains an isometric copy of $\ell_p$, the set $\mathcal{L}(X, Y) \setminus \NA(X,Y)$ is lineable whenever it is non-empty \cite[Proposition 7]{PT}. Recently, in \cite{DFJR23} it is proved among other results that the set $\mathcal{L}(c_0 (\Gamma), Y) \setminus \overline{\NA(c_0(\Gamma), Y)}$ is spaceable, where $\Gamma$ is an infinite set and $Y$ is a strictly convex renorming of $c_0(\Gamma)$. Let us note that the fact that $\mathcal{L}(c_0, Y) \setminus \overline{\NA(c_0, Y)}$ is non-empty was first observed in J. Lindenstrauss's seminal paper \cite{L}.

Before we turn to the context of norm-attaining Lipschitz functions, let us first recall some necessary definitions. Given a pointed metric space $M$ with the distinguished point $0$, let $\lip(M)$ denote the Banach space of real-valued  Lipschitz functions $f:M\rightarrow \mathbb{R}$ such that $f(0)=0$ endowed with the Lipschitz norm
\begin{equation}\label{def:lipnorm}
\|f\|:=\sup\left\{ |S(f,p,q)| :  (p,q) \in \widetilde{M}  \right\}, \quad \text{where } \, S(f,p,q): = \frac{f(q)-f(p)}{d(p,q)}
\end{equation} 
and we denote $\widetilde{A} := \{ (a, b) \in A^2 : a \neq b \}$ for a given set $A$. Note that the choice of the distinguished point $0\in M$ is irrelevant for our purposes, since if $0$ and $0'$ are two different points of $M$, the mapping $\phi:\lip(M)\rightarrow \operatorname{Lip}_{0'}(M)$ given that maps each $f\in\lip(M)$ to $f-f(0')$ is an isometry that preserves the norm and norm-attainment of every function. We say that a Lipschitz function $f\in\lip(M)$ \textit{strongly attains its norm} if the supremum in \eqref{def:lipnorm} is actually a maximum, and denote by $\sna(M)$ the set of strongly norm-attaining Lipschitz functions in $\lip(M)$. The notion of strong norm-attainment has been studied extensively in recent years \cite{CCGMR19, Chiclana22, CGMR21, Godefroy16, JMR23, KMS16}.

A natural element in $\lip(M)^*$ is the evaluation mapping $\delta_x$ at some point $x \in M$, i.e., $\delta_x (f) = f(x)$ for every $f \in \lip(M)$. The space $\calf(M)=\overline{\operatorname{span}}\{\delta_x:\, x\in M\} \subseteq \lip(M)^*$ is known to be an isometric predual of $\lip(M)$, called the \textit{Lipschitz-free space} (also known as Arens-Eells space or transportation cost space). One important universal property of Lipschitz-free space is linearization: given a Lipschitz function $f \in \lip (M)$, there is a unique bounded linear functional $T_f \in  \fm^*$ with $\|T_f\| = \|f\|$ satisfying that $T_f (\delta_x) = f(x)$ for every $x \in M$.  
It is clear that $f\in\lip(M)$ strongly attains its norm at a pair $(p,q)$ if and only if the corresponding bounded linear functional $T_f \in \fm^*$ attains its norm at $\frac{\delta_p - \delta_q}{d(p,q)}$. We will say that $f$ \textit{attains its norm as a functional} on $\fm$ if the corresponding functional $T_f$ attains its norm. For a solid background in Lipschitz functions and Lipschitz-free spaces, we refer the reader to \cite{Godefroy15, Weaver18}.

Investigating infinite-dimensional linear structures within sets of Lipschitz functions that strongly attain their norm has attracted considerable attention, as seen in studies \cite{AMRT23, CJLR2023, DMQR23, KR22}. For instance, if $M$ is an infinite metric space, then the set $\sna(M)$ contains an isomorphic copy of $c_0$ \cite[Main Theorem]{AMRT23}, and if, in addition, $M$ is not uniformly discrete, then $c_0$ is isometrically contained in $\sna(M)$ \cite[Theorem 4.2]{DMQR23}. However, there exist uniformly discrete metric spaces $M$ such that $c_0$ cannot be isometrically contained in $\sna(M)$ \cite[Theorems 4.1 and 4.4]{DMQR23}. Recently, it is observed in \cite[Proposition 3.9]{CJLR2023} that if the set $\sna(M)$ contains an isometric copy of $c_0$, then $\fm$ contains a $1$-complemented copy of $\ell_1$ which provides several examples of $M$ for which $\sna(M)$ does not contain $c_0$ isometrically.

In this paper, we tackle the following two related natural questions.
\begin{enumerate}[label=(Q\arabic*)]
\itemsep0.25em
\item Is $\lip(M) \setminus \sna(M)$ spaceable if $M$ is infinite?\label{q1}
\item Can $\sna(M)$ and $(\lip(M) \setminus \sna(M))\cup\{0\}$ be linear spaces if $M$ is infinite?\label{q2}
\end{enumerate}

In Section \ref{section:SNA}, we answer the first question, \ref{q1}: we show that for every infinite metric space $M$, the set $(\lip(M) \setminus \sna(M))\cup\{0\}$ always contains an isometric copy of $\ell_\infty$. In fact, we prove a stronger result by showing that those functions can be chosen so that their corresponding functionals do not attain their norm as functionals (see Theorem \ref{Main-Theorem-c0-PNA}). As a natural extension of the study, we also investigate the spaceability of the set $\lip(M) \setminus \overline{\sna(M)}$. Namely, we show that in all the known cases where this set is non-empty, $(\lip(M) \setminus \overline{\sna(M)})\cup\{0\}$ actually contains an isometric copy of $\ell_\infty$. Further spaceability results and remarks are considered for several classes of Lipschitz functions that do or do not attain their norms.

In Section \ref{section:linearity}, we give a negative answer to the second question, \ref{q2}, by showing that for every infinite metric space $M$, the sets $\sna(M)$ and $(\lip(M) \setminus \sna(M))\cup \{0\}$ are never linear spaces, contrary to the case of norm-attaining operators (see Theorem \ref{thm:sna_not_linear}). 

Finally, in Section \ref{LipNA}, we provide further spaceability results for sets of Lipschitz functions that do not attain their norms strongly by studying functions that lie in between certain sets of Lipschitz functions that attain their norm in weaker senses from the literature. We provide a diagram of the relations between these sets and show that all the involved regions are spaceable.

\section{Spaceability of non-strongly norm-attaining Lipschitz functions}\label{section:SNA}

In this section, we mainly aim to answer the question \ref{q1}, and further address the spaceability outside the norm-closure of strongly norm-attaining Lipschitz functions. To do so, we first observe the following simple but useful lemma.

\begin{lemma}\label{lemma:special-c0-in-ell-infty}
$\ell_\infty$ contains a closed subspace $Z$ which is an isometric copy of $\ell_\infty$ such that there is no coordinate of maximum modulus for any non-zero element $z \in Z$. That is, there is an isometric isomorphism $T:\ell_\infty\to Z$ such that if $a=(a_n)_{n=1}^\infty\in \ell_\infty$, then
\[ \|T (a)\| > |a_n|\] for all $n\in \mathbb{N}$.
\end{lemma}

\begin{proof}
Let $\{s_n\}_{n=1}^{\infty}$ be the sequence of increasing prime numbers. For each $n\in \bbn$, consider 
\[
z_n = (z_n (j))_{j=1}^\infty \in \ell_\infty
\]
given by 
\[
z_n (s_n^m) := 1-2^{-m} \text{ for each } m\in\bbn,\, \text{ and } \, z_n (j) := 0 \, \text{ otherwise}.
\]
For each $a=(a_n)_n \in \ell_\infty$, consider the vector $z^{(a)} \in \ell_\infty$ defined as the formal series $z^{(a)} :=\sum_n a_n z_n$. In other words, 
\[
z^{(a)} (s_n^m) = a_n (1-2^{-m}) \, \text{ for every } n, m \in \mathbb{N}, \, \text{ and } \, z^{(a)} (j) = 0 \, \text{ for every } j \not\in \{s_n^k : n, k \in \mathbb{N}\}.  
\]
Notice that $\|z^{(a)}\|_\infty =\|a\|_\infty$. However, by construction of $z_n$'s, for a fixed  $m\in\bbn$, we have that $|z^{(a)} (m)|< \|a\|_\infty$. 
\end{proof}

We are now able to show that for any infinite metric space $M$, the set of Lipschitz functions defined on $M$ which do not strongly attain their norm is spaceable, and in fact we show a stronger claim than that. In this document, by the \textit{support} of a Lipschitz function $f \in \Lip(M)$, we mean the set $\supp(f):=\{x\in M:\, f(x)\neq 0\}$.

\begin{theorem}\label{Main-Theorem-c0-PNA}
Given an infinite metric space $M$, 
\[
(\lip(M) \setminus \na (\mathcal{F}(M)))\cup\{0\} \text{ isometrically contains $\ell_\infty$.}
\] 
\end{theorem} 

\begin{proof}
Notice first from the proofs of \cite[Lemma 4, Theorem 5]{CJ17} that there exists $\{f_n\}_n \subseteq \lip(M)$ satisfying the following properties:
\begin{enumerate}[label=(a\arabic*)]
\itemsep0.25em
\item if $n\neq m$, then $\supp(f_n) \cap \supp (f_m) = \emptyset$; \label{a1} 
\item for $a=(a_n)_n \in\ell_\infty$, if $f^{(a)}$ is the pointwise limit of $\sum_{n} a_n f_n$, then $\|f^{(a)}\|=\|a\|_\infty$. \label{a2}
\end{enumerate}
Consider from Lemma \ref{lemma:special-c0-in-ell-infty} the subspace $\{ z^{(a)}: a \in \ell_\infty\}$ of $\ell_\infty$ which is isometrically isomorphic to $\ell_\infty$.
For each $a=(a_n)_n \in \ell_\infty$, let us define $h^{(a)} \in \lip(M)$ to be the pointwise limit of
\[
h^{(a)} := f^{( z^{(a)} )} = \sum_n z^{(a)} (n) f_n. 
\]
Then $\|h^{(a)}\| = \|z^{(a)}\|_\infty = \|a\|_\infty$. 

We claim that $h^{(a)}$ does not attain its norm as a functional on $\mathcal{F}(M)$ if $a\neq 0$. For simplicity, assume that $\|a\|_\infty =1 $. 
Assume to the contrary that $\langle h^{(a)}, \mu \rangle = 1$ for some $\mu\in\fm$ with $\|\mu\|=1$. 
Notice first that given a finitely supported element $\xi \in \fm$, we have 
\begin{equation}\label{eq:finitely_supported}
\langle h^{(a)}, \xi \rangle = \sum_n z^{(a)}(n) \langle f_n, \xi\rangle
\end{equation}
since $\{f_n\}_n$ have disjoint supports. 
Choose a sequence of finitely supported elements $\{ \mu_n \}_n \subseteq \mathcal{F}(M)$ such that $\|\mu_n\|=1$ for all $n\in\bbn$ and $\|\mu-\mu_n\|\rightarrow 0$. 
Passing to a subsequence if needed, we can assume $\|\mu_n-\mu_m\|<\frac{1}{2^n}$ for $m\geq n\geq 1$. Now,
\begin{align}\label{eq22-1}
\nonumber 1=\langle h^{(a)},\mu\rangle = \lim_{k\to\infty} \langle h^{(a)}, \mu_k\rangle &\leq \lim_{k\to\infty} \left( \sum_{n=1}^\infty |z^{(a)}(n)| |\langle f_n, \mu_k\rangle| \right) \\
&\leq \lim_{k\to\infty} \left( \sum_{n=1}^\infty |\langle f_n, \mu_k\rangle| \right)\\ 
&=\lim_{k\to\infty}\left( \sum_{n=1}^\infty \langle \theta_{n,k} f_n, \mu_k\rangle \right)=\lim_{k\to\infty} \langle g_k, \mu_k\rangle \leq 1 \nonumber,
\end{align}
where $g_k:=\sum_{n=1}^\infty \theta_{n,k} f_n$, and $\theta_{n,k}:=\sign(\langle f_n, \mu_k\rangle)$ for each $n,k\in\bbn$. It follows that
\begin{equation}\label{eq22-2}
\lim_{k\to\infty} \sum_{n=1}^\infty (1-|z^{(a)}(n)|) |\langle f_n, \mu_k\rangle|=0.
\end{equation}
On the other hand, for every $l\in\bbn$, 
\begin{align*}
\sum_{n=1}^l |\langle f_n, \mu_k-\mu\rangle| &= \lim_{m\to\infty} \sum_{n=1}^l |\langle f_n, \mu_k-\mu_m\rangle| \\
&= \lim_{m\to\infty} \left\langle \sum_{n=1}^l \tau_{n,k,m} f_n, \mu_k-\mu_m \right\rangle \leq \lim_{m\to\infty} \|\mu_k-\mu_m\|\leq \frac{1}{2^k},
\end{align*}
where $\tau_{n,k,m}:=\sign(\langle f_n, \mu_k-\mu_m\rangle)$ for each $n,k,m\in\bbn$. Thus, letting $l\to\infty$,
\begin{equation}\label{eq22-3}
\sum_{n=1}^\infty |\langle f_n, \mu_k-\mu\rangle| \le \frac{1}{2^k}.
\end{equation}
Note that for each $l,k\in\bbn$,
$$\sum_{n=1}^l |\langle f_n, \mu\rangle| \leq \sum_{n=1}^l |\langle f_n, \mu_k\rangle | + \frac{1}{2^k} = \left\langle \sum_{n=1}^l \alpha_n f_n, \mu_k\right\rangle + \frac{1}{2^k}\leq 1+\frac{1}{2^k},$$
for suitable signs $\alpha_1,\ldots,\alpha_l \in \{-1,1\}$. In particular,
\begin{align}
&\phantom{\leq}\left| \sum_{n=1}^\infty (1-|z^{(a)}(n)|)|\langle f_n, \mu_k\rangle | - \sum_{n=1}^\infty (1-|z^{(a)}(n)|)|\langle f_n, \mu\rangle | \right| \label{eq_difference} \\ 
&  \leq \sum_{n=1}^\infty (1-|z^{(a)}(n)|)\big| |\langle f_n, \mu_k\rangle |-|\langle f_n, \mu\rangle |\big| \nonumber \\
& \leq \sum_{n=1}^\infty (1-|z^{(a)}(n)|)|\langle f_n, \mu_k-\mu\rangle |\leq \sum_{n=1}^\infty |\langle f_n, \mu_k-\mu\rangle |\leq \frac{1}{2^k}. \nonumber
\end{align}
Letting $k\rightarrow \infty$ in \eqref{eq_difference} and using \eqref{eq22-2}, we obtain
$$\sum_{n=1}^\infty (1-|z^{(a)}(n)|)|\langle f_n, \mu\rangle |=0.$$
This shows that $\langle f_n, \mu\rangle=0$ for all $n\in\bbn$, and thus it follows from \eqref{eq22-3} that 
\[
\sum_{n=1}^{\infty} |\langle f_n, \mu_k\rangle| = \sum_{n=1}^\infty |\langle f_n, \mu_k-\mu\rangle| \le \frac{1}{2^k}
\]
for every $k \in \mathbb{N}$. This contradicts \eqref{eq22-1}, so the claim is proven.

Consequently, 
\[
\ell_\infty \stackrel{1}{=} \{h^{(a)} : a \in \ell_\infty\} \subseteq (\lip(M) \setminus \na (\mathcal{F}(M)) ) \cup \{0\}. \qedhere
\] 
\end{proof}

Using the set relation $\sna(M) \subseteq \na(\mathcal{F}(M))$, the following consequence can be obtained directly.

\begin{corollary}\label{cor:Main-Theorem-c0-PNA}
For any infinite metric space $M$, 
\[
(\lip(M) \setminus \sna(M) )\cup\{0\} \text{ isometrically contains $\ell_\infty$.}
\]
\end{corollary}

Next, we focus on spaceability of the set of Lipschitz functions that cannot be approximated by strongly norm-attaining Lipschitz functions. 
Recall that a metric space $M$ is said to be \textit{metrically convex} (resp., \textit{length}) if $d(x,y)$ is the minimum (resp., infimum) of the length of the rectifiable curves joining $x$ and $y$ for every pair of points $(x, y) \in \widetilde{M}$. It is clear that every metrically convex space is a length space, however the converse is not true in general (see \cite[Example 2.4]{IKW07}). It was first observed in \cite[Theorem 2.3]{KMS16} that if $M$ is a metrically convex metric space, then $\sna(M)$ is not norm-dense in $\lip(M)$. Afterwards, it was actually shown that the same result holds for a complete {length} metric space \cite[Theorem 2.2]{CCGMR19}. 
In the following, we show the existence of an isometric copy of $\ell_\infty$ inside the set $(\lip(M) \setminus \overline{\sna(M)})\cup\{0\}$ when $M$ is a complete length space, and an isometric copy of $c_0$ inside the set $(\na(\mathcal{F}(M)) \setminus \overline{\sna(M)} )\cup\{0\} $ when $M$ is a metrically convex space. 

Before we state the theorem, let us recall some facts and lemmas. Given an interval $A \subseteq \bbr$, it is well known that $\lip(A)$ is isometrically isomorphic to $L_{\infty}(A)$ via the identification $f\mapsto f'$, which is defined a.e. (see \cite[Example 3.11]{Weaver18}). More precisely, let $A= [a,b]$. Then the mapping 
\begin{equation}\label{eq:psi}
\Psi : \mathcal{F}([a,b]) \rightarrow L_1 ([a,b]),
\end{equation}
given by $\Psi (m_{x,y}) = |x-y|^{-1} \chi_{[x,y)}$ for every $x<y \in [a,b]$, is an isometric isomorphism. Notice that 
\[
\Phi = (\Psi^{-1})^* : \lip ([a,b]) \rightarrow L_\infty ([a,b]) 
\]
is given by $\Phi (f) = f'$ for every $f \in \lip ([a,b])$, and $\Phi$ is an isometric isomorphism.

\begin{lemma}\label{lemma:equiv-sna_naf}
Let $A$ be any \textup{(}not necessarily bounded\textup{)} interval of $\bbr$ and $f \in \lip(A)$ a non-zero Lipschitz function.
\begin{enumerate}
\itemsep0.25em
\item[\textup{(a)}] \textup{(}\cite[Lemma 2.2]{KMS16}\textup{)} $f$ strongly attains its norm if and only if there exists a subinterval $B\subset A$ with positive measure such that either $f'(x)=\|f\|$ for all $x\in B$ or $f'(x)=-\|f\|$ for all $x\in B$.
\item[\textup{(b)}] \textup{(}\cite[Lemma 2.6]{AAAG07}\textup{)} $f$ attains its norm as a functional on $\calf(A)$ if and only if the set $A_f:=\{x\in A:\, |f'(x)|=\|f\|\}$ has positive Lebesgue measure.
\end{enumerate} 
\end{lemma}

Given $p\in M$ and $r>0$, we denote by $B(p, r):=\{q\in M:\, d(p,q)\leq r\}$ the closed ball centered at $p$ with radius $r$.

\begin{lemma}\label{lemma:new-extension-lemma-[0,1]}
Let $M$ be an infinite pointed metric space containing distinct points $\{p_n\}_{n=1}^{\infty}$ and $\{q_n\}_{n=1}^\infty$ such that $d(p_n, p_m) > d(p_n, q_n) + d(p_m, q_m)$ for every $(n,m) \in \widetilde{\mathbb{N}}$. 
Suppose that $\{g_n\}_{n=1}^{\infty}\subset \lip(M)$ satisfies that $\|g_n\| =1$ and 
\begin{equation}\label{eq:g_n_bounded_by_cone}
|g_n(p)| \leq \max \{0, d(p_n,q_n)-d(p_n,p)\} \, \text{ for all } n \in \mathbb{N} \text{ and } p \in M.  
\end{equation}
Then given $\lambda=(\lambda_n)_n\in \ell_\infty \setminus\{0\}$, the element $g^{(\lambda)}:= \sum_{n=1}^{\infty} \lambda_n g_n$ \textup{(}pointwise limit\textup{)} satisfies that $\|g^{(\lambda)}\| = \|\lambda\|_\infty$. Moreover, if $g_n \not\in \sna(M)$ for every $n\in\mathbb{N}$, then $g^{(\lambda)} \not\in \sna(M)$.  
\end{lemma}

\begin{proof}
Given $(p,q)\in\widetilde{M}$, we shall estimate the value $|S(g^{(\lambda)}, p, q)|$. Note from \eqref{eq:g_n_bounded_by_cone} that $\supp (g_n) \subseteq B(p_n, d(p_n,q_n))$ for each $n \in \mathbb{N}$. Let us first consider the case when $p \in B(p_n, d(p_n,q_n))$ and $q \in  B(p_m, d(p_m,q_m))$ for some $n < m$. Then 
it follows that,  if $\max \{|\lambda_n|, |\lambda_m|\}\neq0$,
\begin{align}\label{lem:ga1}
|S(g^{(\lambda)}, p, q)| &\leq \frac{|\lambda_n|\left( d(p_n, q_n)-d\left(p_n, p \right) \right) + |\lambda_m|\left( d(p_m, q_m)-d(p_m, q) \right) }{d(p,q)}  \\
&< \max \{|\lambda_n|, |\lambda_m|\}.\nonumber
\end{align}
Next, we consider the case when $p, q \in B(p_n, d(p_n,q_n))$ for some $n \in \mathbb{N}$. Then it is clear that 
\begin{equation}\label{lem:ga2}
S(g^{(\lambda)}, p,q) = \lambda_{n} S(g_n, p, q).     
\end{equation}
As the other cases are easily handled; we conclude that $\|g^{(\lambda)}\| = \|\lambda\|_\infty$. The second assertion follows from \eqref{lem:ga1} and \eqref{lem:ga2}. 
\end{proof}

Finally, we are ready to prove the promised result with the following lemma.

\begin{lemma}[\mbox{\cite[Lemma 2.1]{CCGMR19}}] \label{lem:2.1}
Let $M$ be a metric space, $f \in \sna (M)$ which strongly attains its norm at $(p,q) \in \widetilde{M}$. Let $\eps >0$ and let $\alpha_\eps$ be a rectifiable curve in $M$ joining $p, q$ such that $\textup{length}(\alpha_\eps)\leq d(p,q)+\eps$, where $\textup{length}(\alpha_\eps)$ is the length of $\alpha_\eps$. 
Then for any $z_1, z_2 \in \alpha_\eps$, we have that $|f(z_1)-f(z_2)| \geq \|f\| (d(z_1,z_2)-\eps)$. 
\end{lemma}

\begin{theorem}\label{thm:main_sna_c0} \leavevmode
\begin{enumerate}
\itemsep0.25em  
\item[\textup{(a)}] If $M$ is a complete length metric space, then 
\[
(\lip (M) \setminus\overline{\sna(M)})\cup\{0\}  \text{ isometrically contains $\ell_\infty$.}
\]
\item[\textup{(b)}] If $M$ is a metrically convex metric space, then 
\[
(\na(\mathcal{F}(M))  \setminus\overline{\sna(M)})\cup\{0\} \text{ isometrically contains $c_0$.}
\]
\end{enumerate}

\end{theorem}

\begin{proof}
(a): Since $M$ is arc-connected, there exist disjoint sequences of distinct points $\{p_n\}_{n=1}^{\infty}\subset M$ and $\{q_n\}_{n=1}^{\infty}\subset M\setminus \{p_n\}_{n=1}^{\infty}$ such that $d(p_n, p_m)> d(p_n, q_n) + d(p_m, q_m)$ for all $(n,m)\in\widetilde{M}$. Let $u_n:M\rightarrow [0,\, d(p_n, q_n)]$ be the surjective norm-one Lipschitz function given by
$$
u_n(p):= \max\{0,\, d(p_n, q_n) - d(p_n, p)\},\quad \text{for all }p\in M.
$$ 
Let $A_n$ be a nowhere-dense subset of $[0,\, d(p_n,q_n)/2]$ with positive measure, and let $B_n=A_n+d(p_n,q_n)/2:=\{x+d(p_n,q_n)/2:\, x\in A_n\}$. Define $f_n\in\lip([0,d(p_n, q_n)])$ as the Lipschitz function such that 
\[f_n'=\chi_{A_n} - \chi_{B_n} \ \ \ \text{a.e.},\]  where $\chi_A$ is the characteristic function on a set $A$ which has the value 1 on $A$ and 0 elsewhere.

For each $n\in\bbn$, let $g_n:=f_n\circ u_n$, and note that $\| g_n\|\leq \|f_n\|= 1$. We will show that $\| g_n\|=1$. Fix $n\in\bbn$ and let $x,y \in [0,  d(p_n, q_n)]$ be given. For each $t>0$, there exists some $0<s=s(t)<t$ and some rectifiable curve $\gamma_n^s:[0, (1+s) d(p_n,q_n) ] \rightarrow M$ joining $p_n$ and $q_n$ and with $\operatorname{length}(\gamma_n^s)=(1+s) \, d(p_n, q_n)$. Find $p,q \in \gamma_n^s ([0, (1+s)d(p_n, q_n)])$ such that $u_n(p)=x$ and $u_n(q)=y$. Note that 
\begin{align}
&d(p_n,q_n) \leq d(p_n,q)+d(q,q_n), \nonumber \\
&d(p_n,p)+d(p,q)+d(q,q_n)\leq \operatorname{length}(\gamma_n^s)=  (1+s)d(p_n, q_n); \label{pn_p_q_qn}
\end{align}
thus
\[
d(p_n,q)-d(p_n,p) \geq d(p,q)-s\, d(p_n, q_n). 
\]
Changing the role of $p$ and $q$, we conclude that 
\[
|d(p_n, p)-d(p_n,q)| \geq d(p,q)-s\, d(p_n, q_n). 
\]
Moreover, by \eqref{pn_p_q_qn}, we have 
\[
d(p_n,q_n)-d(p_n, p) \geq d(p,q)+d(q,q_n)-s \,d(p_n,q_n)  \geq d(p,q) -s \,d(p_n,q_n)  
\]
and similarly,  
\[
d(p_n,q_n)-d(p_n,q)   \geq d(p,q) -s \,d(p_n,q_n).   
\]
Therefore, we can observe that 
\[
d(p,q)- s\, d(p_n, q_n)\leq |u_n(p)-u_n(q)|=|x-y|. 
\]
Having also in mind that
\[
d(p,q)- s\, d(p_n, q_n)\leq |x-y| = |u_n(p)-u_n(q)| \leq d(p,q) \leq (1+s) d(p_n, q_n)
\]
observe 
\begin{align*}
\|g_n\| \geq |S(g_n,p,q)| = \frac{|f_n(x)-f_n(y)|}{d(p,q)} &\geq \frac{|f_n(x)-f_n(y)|}{|x-y| +  s\,d(p_n, q_n)} \\
&= |S(f_n,x,y)| \left(1 - \frac{s \, d(p_n, q_n)}{|x-y|+ s\,d(p_n, q_n)}\right) \\
&\geq |S(f_n,x,y)| \left(1 - \frac{s \, d(p_n, q_n)}{ (1+s)d(p_n, q_n)+s\,d(p_n, q_n)}\right) \\
&= |S(f_n,x,y)| \left( 1 - \frac{s}{1+2s} \right).
\end{align*}
Letting $t\rightarrow 0$ (so, $s\rightarrow0$), we conclude that $\|g_n\| \geq |S(f_n,x,y)|$. As $x,y \in [0,d(p_n,q_n)]$ are chosen arbitrarily, we conclude that $\|g_n\|=1$. 

Moreover, by construction,  $\supp(g_n)\subset B(p_n,\, d(p_n, q_n))$. Let $\lambda=(\lambda_n)_n\in \ell_\infty \setminus \{0\}$, and let $g^{(\lambda)}:=\sum_n \lambda_n g_n$. 

First, for each $n \in \mathbb{N}$ and $p \in M$, observe that
$$ 
|g_n(p)| \leq |u_n(p)|= \max\left\{0, d(p_n, q_n) - d(p_n, p)\right\}.
$$
Then Lemma \ref{lemma:new-extension-lemma-[0,1]} shows that $\|g^{(\lambda)}\| = \|\lambda\|_\infty$.

Finally, we will show that $g^{(\lambda)}\notin \overline{\sna(M)}$. Let $h \in \sna(M)\setminus \{0\}$ which strongly attains its norm at $(p,q)\in\widetilde{M}$. Suppose that $\|g^{(\lambda)}-h\| < \frac{\|\lambda\|_\infty}{2}$. Notice that $\|h\| >  \frac{\|\lambda\|_\infty}{2} $. Note that if $g^{(\lambda)}(p)=g^{(\lambda)}(q)$, then 
\[
\|g^{(\lambda)}- h\| \geq \|h\| > \frac{\|\lambda\|_\infty}{2},
\]
which is a contradiction. Thus, we assume that $g^{(\lambda)}(p) \neq g^{(\lambda)}(q)$. Without loss of generality, let us say $g^{(\lambda)}(q) \neq 0$. Then there is $n\in\bbn$ such that $\lambda_n \neq 0$ and $u_n(q) > 0$, which implies that $q\in B(p_n,\, d(p_n, q_n))$. Also, up to rellabeling the points if needed, we can assume that $u_n(q)>u_n(p)\geq 0$. 
Pick $r\in B(p_n,\, d(p_n, q_n))$ such that $0< u_n (r) < u_n(q)$. Since $A_n$ is nowhere dense, there exist $s_0<t_0$ such that $[s_0, t_0]\subset ]u_n(r), u_n(q)[$ and $f_n'$ is identically zero on $[s_0, t_0]$. 
Take $\eps_0>0$ sufficiently small so that 
\[
0< \eps_0 < |t_0-s_0| \frac{\|h\| - \|g^{(\lambda)}- h\|}{\|h\|}
\]
and a rectifiable curve $\alpha_{\eps_0} : [0, d(p,q)+\eps_0] \rightarrow M$ joining $p$ and $q$ such that $\text{length}(\alpha_{\eps_0}) \leq d(p,q) + \eps_0$.
Find $z_0, w_0 \in \alpha_{\eps_0} ([0, d(p,q)+\eps_0])\cap B(p_n,\, d(p_n, q_n))$ such that $u_n(z_0)=s_0$ and $u_n(w_0)=t_0$. 
Then 
\[
|t_0-s_0| = |u_n(w_0)-u_n(z_0)| \leq d(w_0, z_0).
\] 
By Lemma \ref{lem:2.1}, we obtain $|h(z_0)-h(w_0)| \geq \|h\| ( d(z_0,w_0) -  \eps_0)$.
Thus, 
\begin{align*}
|g^{(\lambda)}(z_0)-g^{(\lambda)}(w_0)| &\geq |h(z_0)-h(w_0)| - \|g^{(\lambda)}-h\| d(z_0, w_0) \\
&\geq \|h\| (d(z_0,w_0)-\eps_0) - \|g^{(\lambda)}-h\| d(z_0,w_0) \\
&\geq \left(\|h\|-\|g^{(\lambda)}-h\| - \frac{\eps_0 \|h\|}{|t_0 - s_0|} \right) d(z_0,w_0) \\ 
&> \left(\|h\|-\|g^{(\lambda)}-h\| -( \|h\| -\|g^{(\lambda)}-h\|) \right) d(z_0,w_0)=0. 
\end{align*}
Since $z_0,w_0\in B(p_n, d(p_n, q_n))$, this shows that $f_n (s_0) \neq  f_n (t_0)$, which contradicts that $f_n'$ is identically zero on $[s_0,t_0]$. Therefore, we conclude that $\|g^{(\lambda)}-h\| > \frac{\|\lambda\|_\infty}{2}$ for every $h \in \sna(M)$.

(b): Next, suppose that $M$ is metrically convex. Consider the functions $u_n:M \rightarrow [0,d(p_n,q_n)]$, $f_n \in \lip ([0,d(p_n,q_n)])$, and $g_n = f_n \circ u_n \in \lip(M)$ for $n \in \mathbb{N}$ as above. 
To finish the proof, it is enough to see that given $\lambda=(\lambda_n)_n\in c_0 \setminus \{0\}$, the element $g^{(\lambda)}=\sum_n \lambda_n g_n \in \mathcal{F}(M)^*$ attains its norm at a norm-one element in $\fm$.
Fix $n \in \mathbb{N}$ so that $|\lambda_n|=\|\lambda\|_\infty$. 
Due to metric convexity of $M$, we can consider $\gamma_n : [0,d(p_n, q_n)] \rightarrow M$ such that $\gamma_n (0)=q_n,\, \gamma_n (d(p_n,q_n))=p_n$, and 
\[
|x-y| = d(\gamma_n (x), \gamma_n (y)), \quad x,y \in [0,d(p_n,q_n)].
\]
That is, $\gamma_n ([0,d(p_n,q_n)])$ is an isometric copy of $[0,d(p_n,q_n)]$ in $M$. For $t \in [0, d(p_n,q_n)]$, note from the definition of $u_n$ that 
\begin{align*}
u_n (\gamma_n (t)) = d(p_n,q_n)-d(p_n, \gamma_n (t))  = d(q_n, \gamma_n(t)) = d( \gamma_n(0), \gamma_n(t)) = |0-t| = t. 
\end{align*}
Observe from Lemma \ref{lemma:equiv-sna_naf} that $f_n \in \lip([0,d(p_n,q_n)]) = \mathcal{F}( [0,d(p_n,q_n)] )^*$ attains its norm at $\Psi_n^{-1} (h_n) \in \mathcal{F}([a,b])$, where $\Psi_n : \mathcal{F}([0,d(p_n,q_n)]) \rightarrow L_1 ([0,d(p_n,q_n)])$ is the canonical isometric isomorphism as in \eqref{eq:psi} and 
$h_n \in L_1([0,d(p_n,q_n)])$ is given by  
\[
h_n (t) =\frac{1}{2}\left(\frac{1}{m(A_{n})}\chi_{A_n} - \frac{1}{m(A_{n})}\chi_{B_n}\right)
\]

Consider the extension $\widehat{\gamma}_n :  \mathcal{F}([0,d(p_n,q_n)]) \rightarrow \mathcal{F}(M)$ of the isometry $\gamma_n : [0,d(p_n,q_n)] \rightarrow M$. 
Now, observe that  
\begin{align*}
\langle g^{(\lambda)}, \, \widehat{\gamma}_n (\psi_n^{-1}(h_n))\rangle &= \langle (\widehat{\gamma}_n)^* (g^{(\lambda)}), \, \psi_n^{-1}(h_n)\rangle \\
&= \langle g^{(\lambda)} \circ \gamma_n, \,\psi_n^{-1}(h_n)\rangle \\
&= \lambda_n \langle f_n\circ u_n \circ \gamma_n, \psi_n^{-1}(h_n)\rangle= \lambda_n \langle f_n, \psi_n^{-1}(h_n)\rangle = \lambda_n;
\end{align*}
hence $g^{(\lambda)}$ attains its norm at $\widehat{\gamma}_n (\psi_n^{-1}(h_n))\rangle$ as desired.
\end{proof}

Recall that an $\mathbb{R}$-tree is a metric space $M$ satisfying the following two conditions:
\begin{enumerate}
\itemsep0.25em
\item For any $a, b \in M$, there exists a unique isometry $\phi : [0,d(a,b)]\rightarrow M$ such that $\phi(0)=a$ and $\phi(d(a,b)) = b$. 
\item Any one-to-one continuous mapping $\varphi : [0,1] \rightarrow M$ has same range as the isometry $\phi$ associated to the points $a=\varphi(0)$ and $b=\varphi(1)$. 
\end{enumerate} 

We say that a subset $A$ of an $\mathbb{R}$-tree $M$ is \textit{measurable} whenever $\phi_{xy}^{-1} (A)$ is Lebesgue measurable for any $x,y\in M$, where $\phi_{xy}$ denotes the unique isometry associated to $x$ and $y$ as above. If $A$ is measurable and $S$ is a segment $[x,y]$ (that is, the image of $\phi_{xy}$), we write $\lambda_S (A)$ for $\lambda(\phi_{xy}^{-1} (A))$ where $\lambda$ is the Lebesgue measure on $\mathbb{R}$. We denote by $\mathcal{R}$ the set of those subsets of $M$ which can be written as a finite union of disjoint segments, and for $R = \cup_{k=1}^n S_k$ in $\mathcal{R}$ (disjoint union), we put $\lambda_R (A) := \sum_{k=1}^n \lambda_{S_k} (A)$. Now, 
\[
\lambda_M (A)  := \sup_{R \in \mathcal{R}} \lambda_R (A)
\]
defines the \textit{length measure} of a measurable subset $A$ of $M$. 

It is also observed in \cite[Theorem 2.3]{CCGMR19} that $\sna(M)$ is not norm-dense in $\lip(M)$ when $M$ is a closed subset of a pointed $\mathbb{R}$-tree with positive length measure containing the distinguished point. We prove that it is actually possible to embed isometrically $\ell_\infty$ (resp., $c_0$) into the complement of the norm-closure of $\sna(M)$ in $\lip(M)$ (resp., $\na(\mathcal{F}(M))$).

\begin{theorem}
Let $M$ be a closed subset of a pointed $\mathbb{R}$-tree $T$ containing the distinguished point. If $M$ has a positive length measure, then 
    \[
( \lip (M)   \setminus\overline{\sna(M)})\cup\{0\} \text{ isometrically contains $\ell_\infty$.}
\]
Moreover, 
\[
(\na(\mathcal{F}(M) )  \setminus\overline{\sna(M)})\cup\{0\} \text{ isometrically contains $c_0$.}
\]
\end{theorem}

\begin{proof}
As $M$ has positive length measure, there is a segment $S=[x_0,y_0] \subseteq T$ such that $\lambda_T (M \cap S)>0$, where $\lambda_T$ is the length measure on $T$ defined as in \cite{Godard10}. We distinguish two cases:

\underline{Case 1}: Assume that there is a segment $[u_0, v_0] \subseteq M \cap S$. By Theorem \ref{thm:main_sna_c0}, there exists $\{f_n\}_n \subseteq \lip ([u_0,v_0])$ such that if $\lambda=(\lambda_n)_n \in \ell_\infty$, then $f^{(\lambda)} := \sum_{n} \lambda_n f_n$ satisfies that $\|f^{(\lambda)}\| = \|\lambda\|_\infty$ and $\|f^{(\lambda)} - g \| \geq \frac{\|\lambda\|_\infty}{2}$ for every $g \in \sna([u_0,v_0])$. Moreover, if $\lambda \in c_0$, then $f^{(\lambda)} \in  \na (\mathcal{F}([u_0,v_0])$. Consider as in \cite[Theorem 2.3]{CCGMR19} the metric projection $\pi_0 : T \rightarrow [u_0,v_0]$ which satisfies that 
\begin{equation}\label{eq:metric_projection}
    d(x,y)=d(x,\pi_0(x)) + d(\pi_0 (x),y) \, \text{ for all } x \in T, y \in [u_0,v_0]. 
    \end{equation}

For each $n \in \mathbb{N}$, define $\widetilde{f_n} \in \lip(M)$ by $\widetilde{f_n}(p)=f_n(\pi_0(p))$ for every $p \in M$. Let $\lambda=(\lambda_n)_n \in \ell_\infty$ be given and put $\widetilde{f^{(\lambda)}} := \sum_n \lambda_n \widetilde{f_n} $. Since $d(\pi_0(p), \pi_0 (q)) \leq d(p,q)$ for every $(p,q) \in \widetilde{M}$, we observe that $\| \widetilde{f^{(\lambda)}}\| = \|\lambda\|_\infty$. We will show that $\widetilde{f^{(\lambda)}}\notin\overline{\sna(M)}$. Assume to the contrary that there is $g \in \sna(M)$ such that $\| \widetilde{f^{(\lambda)}} - g \| < \frac{\|\lambda\|_\infty}{2}$. Say $S(g,x,y)=\|g\|$ for some $(x,y) \in \widetilde{M}$. If $\pi_0 (x)=\pi_0 (y)$, then this implies that 
\begin{equation}\label{eq:same_pi}
\frac{\|\lambda\|_\infty}{2} \geq \left|S \left( \widetilde{f^{(\lambda)}} - g, x,y \right) \right| = S(g,x,y) >  \frac{\|\lambda\|_\infty}{2},
\end{equation}
which is a contradiction. Thus, $\pi_0 (x) \neq \pi_0 (y)$; hence $S(g \vert_{[u_0,v_0]}, \pi_0 (x), \pi_0 (y)) = \|g\|$. This implies that 
\[
\frac{\|\lambda\|_\infty}{2} > \left\| \widetilde{f^{(\lambda)}} - g \right\| \geq \left\| f^{(\lambda)} - g \vert_{[u_0,v_0]} \right\| \geq \frac{\|\lambda\|_\infty}{2}, 
\]
which is again a contradiction. Thus, we conclude that $\| \widetilde{f^{(\lambda)}} -g\| \geq \frac{\|\lambda\|_\infty}{2}$ for every $g \in \sna(M)$. In particular, $\widetilde{f^{(\lambda)}} \not\in \overline{\sna(M)}$.

To see that $\widetilde{f^{(\lambda)}} \in \na(\mathcal{F}(M))$ provided that $\lambda \in c_0 \setminus \{0\}$, let us take $\mu \in \mathcal{F}([u_0,v_0])$ such that 
$\langle f^{(\lambda)}, \mu \rangle = \|\lambda\|_\infty$. If we consider an isometric embedding $\gamma : [u_0, v_0] \rightarrow M$, then it is routine to check that $\widetilde{f^{(\lambda)}}$ attains its norm as a functional at $\widehat{\gamma} (\mu)$, where $\widehat{\gamma} : \mathcal{F}([u_0,v_0]) \rightarrow \mathcal{F}(M)$ is the canonical extension of $\gamma$; consequently, $\widetilde{f^{(\lambda)}} \in \na (\mathcal{F}(M)) \setminus \overline{\sna(M)}$.

\underline{Case 2}: Assume that no segment is contained in $M \cap S$. Let $\{[u_n, v_n]\}_{n=1}^\infty \subseteq S$ be a sequence of pairwise disjoint closed intervals such that $\lambda_T(M \cap [u_n,v_n])>0$ for every $n \in \mathbb{N}$. For each $n \in \mathbb{N}$, pick $p_n \in [u_n,v_n]$ such that $\lambda_T (M \cap [u_n, p_n] ) = \lambda_T (M \cap [p_n, v_n])$ which in turn implies that  
\[
\lambda_T (M \cap [u_n, p_n] ) = \lambda_T (M \cap [p_n, v_n]) = \frac{1}{2} \lambda_T (M \cap [u_n, v_n]) > 0.
\]
Define $h_n : [u_n, v_n] \rightarrow \mathbb{R}$ by 
\[
h_n (p) = \lambda_T (M \cap [u_n, p_n]) - \lambda_T (M \cap [p_n, p]) \, \text{ for all } p \in [u_n,v_n].
\]
Observe that $h_n$ is a norm-one Lipschitz function and $h_n(u_n) = h_n(v_n)=0$. For each $n \in \mathbb{N}$, let $\pi_n : T \rightarrow [u_n,v_n]$ be the metric projection as in \eqref{eq:metric_projection}, and define $\widetilde{h_n} : M \rightarrow \mathbb{R}$ by $\widetilde{h_n}(p)=h_n (\pi_n (p))$ for every $p \in M$. Notice that $\|\widetilde{h_n}\|=1$ for every $n \in \mathbb{N}$, and 
\[
\widetilde{h_n} (\pi_m (p)) = 0
\]
for every $p \in M$ provided that $n \neq m$. 

Let $\lambda = (\lambda_n)_n \in \ell_\infty \setminus\{0\}$ and put $\widetilde{h^{(\lambda)}} = \sum_n \lambda_n \widetilde{h_n}$. It is clear that $\|\widetilde{h^{(\lambda)}}\| \geq \|\lambda\|_\infty$. To see that it is an equality, let $(p,q) \in \widetilde{M}$ be given. Note that if $p \not\in ]u_n, v_n[$ for every $n \in \mathbb{N}$, then $\widetilde{h^{(\lambda)}}(p) = 0$. 
If $p \in ]u_n, v_n[$ for some $n \in \mathbb{N}$, then 
\begin{align*}
|\widetilde{h^{(\lambda)}} (p)| &= |\lambda_n| | \lambda_T (M \cap [u_n, p_n]) - \lambda_T (M \cap [p_n, p]) | 
\end{align*}
Consequently, 
\begin{align*}
&|\widetilde{h^{(\lambda)}} (p)| \leq |\lambda_n| \lambda_T (M \cap [u_n, p]) \leq |\lambda_n| \lambda_T (M \cap [p, v_n]) \text{ if } p \in [u_n,p_n]  \\ 
&|\widetilde{h^{(\lambda)}} (p)| \leq |\lambda_n| \lambda_T (M \cap [p, v_n]) \text{ if } p \in [p_n, v_n];
\end{align*}
hence $|\widetilde{h^{(\lambda)}} (p)| \leq |\lambda_n| \lambda_T (M \cap [p, v_n])$ for $p \in [u_n, v_n]$. 
Suppose that $p \in ]u_n, v_n[$ and $q \in ]u_m, v_m[$ for some $n, m \in \mathbb{N}$. By changing $p$ and $q$ if necessary, we may assume that $[v_n, u_m] \subseteq [u_n, v_m]$. 
Observe that 
\begin{align*}
    |S( \widetilde{h^{(\lambda)}}, p,q)| 
    \leq \frac{|\lambda_n| \lambda_T (M \cap [p, v_n]) + |\lambda_m| \lambda_T (M \cap [u_m, q])  }{d(p,q)} \leq \max \{ |\lambda_n| , |\lambda_m| \}. 
\end{align*}

The other cases, i.e., 
\begin{itemize}
\itemsep0.25em
\item $p,q  \in ]u_n, v_n[$ for some $n \in \N$;  \item $p \in ]u_n, v_n [$ for some $n \in \N$ while $q \not\in ]u_m, v_m[$ for every $m \in \N$
\end{itemize} 
also can be handled similarly; so we conclude that $\| \widetilde{h^{(\lambda)}}\| = \|\lambda\|_\infty$.

Next, we claim that $\widetilde{h^{(\lambda)}} \not\in \overline{{\sna(M)}}$. To this, assume again to the contrary that $\| \widetilde{h^{(\lambda)}} - g \| < \frac{\|\lambda\|_\infty}{2}$ for some $g \in \sna(M)$. Let $n \in \N$ be given and say $|S(g,p,q)|=\|g\|$ for some $(p,q)\in\widetilde{M}$. If $\pi_n (p) = \pi_n (q)$ for every $n \in \N$, then as in \eqref{eq:same_pi}, we have
\[
\|\widetilde{h^{(\lambda)}}-g\| \geq |S(g,p,q)|=\|g\| > \frac{\|\lambda\|_\infty}{2},
\]
which is a contradiction. So, we suppose that $\pi_n(p)\neq\pi_n (q)$ for some $n \in \N$. Since there is no segment in $M\cap S$, we have that $[\pi_n(p), \pi_n (q)] \not\subseteq M \cap [u_n, v_n]$. Thus, we can find $(z_n, w_n) \in \widetilde{M}$ such that $[z_n, w_n] \subseteq ]\pi_n(p), \pi_n (q)[ \, \setminus (M\cap [u_n,v_n])$. Notice that 
\begin{align*}
\widetilde{h_n} (z_n) = h_n (z_n) &= \lambda_T (M \cap [u_n, p_n]) - \lambda_T (M \cap [p_n, z_n]) \\
&= \lambda_T (M \cap [u_n, p_n]) - \lambda_T (M \cap [p_n, w_n]) = h_n (w_n) = \widetilde{h_n} (w_n).
\end{align*} 
It follows that 
\[
\| \widetilde{h^{(\lambda)}} - g \| \geq |S(g, z_n, w_n)| = \|g\| > \frac{\|\lambda\|_\infty}{2},
\]
which is a contradiction. 

Finally, fix $\lambda=(\lambda_n)_n \in c_0 \setminus \{0\}$. We claim that $\widetilde{h^{(\lambda)}} \in \na (\mathcal{F}(M))$. Fix $n \in \N$ so that $|\lambda_n| = \|\lambda\|_\infty$. Consider the $\bbr$-tree $T_n:=[u_n, v_n]$ with distinguished point $0':=u_n$ and its subset $A: = M \cap T_n$. Recall from \cite[Definition 2.2]{Godard10} the positive measure $\mu_A$ defined by 
\[
\mu_A = \lambda_{T_n} \vert_A + \sum_{a \in A} L(a) \delta_a,
\]
where $L(a)=\inf_{x \in A \cap [0', a[ } d(a,x)$ and $\delta_a$ is the Dirac measure on $a$. By definition, note that $h_n \vert_{A} \in  \operatorname{Lip}_{0'} (A)$ has norm one. It is observed in \cite[Theorem 3.2]{Godard10} that 
\[
\Phi : L_\infty (A, \mu_{A})\rightarrow \operatorname{Lip}_{0'} (A)
\]
given by $\Phi(g)(a) = \int_{[0',a]} g \, d\mu_A$ is a weak$^*$-weak$^*$ continuous surjective linear isometry. Let $\Psi :L_1 (A, \mu_A) \rightarrow \mathcal{F}(A)$ be the linear isometry so that $\Psi^* = \Phi^{-1}$. Set $\widecheck{A}:= \{ a \in A: L(a) >0\}$. As $\lambda_{T_n} (A)>0$ and $\widecheck{A}$ is at most countable, we observe that the set $A_0 := A \setminus \widecheck{A}$ is non-empty. Take $\xi \in L_1 (A, \mu_A)$ defined by 
\[
\xi (p) = \begin{cases}
1/\lambda_{T_n} (A_0), \quad &\text{if } p \in A_0 \cap [u_n, p_n] \\
-1/\lambda_{T_n} (A_0),\quad &\text{if } p \in A_0 \, \cap \, ]p_n, v_n],\\
0,\quad &\text{else}.
\end{cases}
\]

It can be checked by definition that 
\[
\Phi ( \chi_{M\cap [u_n, p_n]\setminus \widecheck{A}}-\chi_{M\cap [p_n, v_n]\setminus \widecheck{A}} ) (a)= h_n (a)\, \text{ for each } a \in A,
\]
that is, $\Phi^{-1}(h_n) = \chi_{M\cap [u_n, p_n]\setminus \widecheck{A}}-\chi_{M\cap [p_n, v_n]\setminus \widecheck{A}}$.  
Thus, we can observe that 
\[
\langle h_n \vert_{A}, \Psi (\xi) \rangle = \langle \Phi^{-1} (h_n \vert_{A}) , \xi \rangle = 1; 
\]
hence $h_n \vert_{A} \in \na (\mathcal{F}(A))$. Considering an isometric embedding $\gamma$ from $A$ into $M$ as above, we can conclude that $\langle \widetilde{h^{(\lambda)}}, \widehat{\gamma} (\Psi (\xi)) \rangle = \lambda_n$; 
so $\widetilde{h^{(\lambda)}}$ attains its norm at $\widehat{\gamma} (\Psi (\xi))$. 
\end{proof}

Let us examine another negative result along this line: the set $\sna(\Gamma)$ is not norm-dense in $\lip(\Gamma)$ when $\Gamma$ is the unit sphere in $2$-dimensional Euclidean space \cite[Theroem 2.1]{CGMR21}. Furthermore, it is generalized in \cite[Theorem 1.2]{Chiclana22} that the same is true when $\Gamma$ is the image of an injective $C^1$ curve $\alpha: [0,1]\rightarrow E$, where $E$ is a normed space and $\alpha'$ is non-identically zero. We will prove now that, in fact, the complement the norm-closure of $\sna(\Gamma)$ union zero contains an isometric copy of $\ell_\infty$.

\begin{theorem}\label{thm:C1}
Let $E$ be a normed space, $\alpha : [0,1]\rightarrow E$ an injective $C^1$ curve with $\alpha'$ non-identically zero, and $\Gamma \subseteq E$ its range. Then 
\[
(\lip(\Gamma) \setminus\overline{\sna(\Gamma)})\cup\{0\} \text{ isometrically contains $\ell_\infty$.}
\]
\end{theorem}

Before we present the proof, let us recall some lemmas from \cite{Chiclana22}.

\begin{lemma}\label{lem:C1}
Let $f : (0,1]\rightarrow [0,1]$ be a function, $E$ a normed space, $\rho >0$, and $\alpha : [0,\rho]\rightarrow E$ a $C^1$ curve parametrized by arc length. 
\begin{enumerate}
\itemsep0.25em
\item[\textup{(1)}] \textup{(\mbox{\cite[Theorem 1.1]{Chiclana22}})} The following are equivalent:
\begin{enumerate}
\itemsep0.25em
\item[\textup{(a)}] $\lim_{x\rightarrow 0^+} f(x)=1$ and $\inf_{x \in (0,1]} f(x) >0$. 
\item[\textup{(b)}] There exists a measurable set $C \subseteq [0,1]$ of positive measure satisfying 
\[
\frac{|C \cap I|}{|I|} < f(|I|) 
\]
for every nontrivial interval $I \subseteq [0,1]$. 
\end{enumerate}
\item[\textup{(2)}] \textup{(\mbox{\cite[Lemma 3.1]{Chiclana22}})} Consider the function $g: (0,\rho] \rightarrow \mathbb{R}$ given by 
\[
g(x) =  \left\{ \frac{\|\alpha(t)-\alpha(s)\|}{|t-s|} : t, s \in [0,\rho], \, |t-s|=x \right\}, \quad x \in (0,\rho]. 
\]
Then $\lim_{x\rightarrow 0^+} g(x) = 1$. \label{lem:C1_2}
\end{enumerate} 
\end{lemma} 

\begin{proof}[Proof of Theorem \ref{thm:C1}]
Since $\alpha'$ is continuous and non-identically zero, we can take a non-trivial interval $J_0 \subseteq [0,1]$ so that $\alpha'(t) \neq 0$ for every $t \in J_0$. Reparametrize $\alpha$ in $J_0$ with respect to arc length. Thanks to \eqref{lem:C1_2} in Lemma \ref{lem:C1}, $J_0$ can be chosen small enough so that $|t-s|/2 \leq \| \alpha(t)-\alpha(s) \| \leq |t-s|$. Up to a change of variables, write $J_0 = [0,\rho]$ for some $0<\rho<1$. Note that $\Gamma_0 : = \{ \alpha(t) : t \in J_0 \}$ is arc-connected; so we can consider $\{q_n\}_{n=1}^\infty \subseteq \Gamma_0$ and $(r_n)_n \subseteq (0,\infty)$ so that $d(q_n,q_m) > 2 (r_n + r_m)$ for every $n\neq m$, and we choose $r_n$'s small enough so that $B(q_n, r_n) \cap \Gamma$ has only one connected component. Taking a subsequence if necessary, we assume that the distinguished point $0$ is not in $\cup_{m=1}^\infty B(q_m, r_m)$. Take $u_n <v_n$ in $J_0$ so that $\alpha([u_n,v_n]) \subseteq B(q_n, r_n/40)$. 

For each $n \in \mathbb{N}$, consider $g_n : (0, v_n-u_n] \rightarrow \mathbb{R}$ by 
\[
g_n (x) =  \inf \left\{ \frac{\|\alpha(t)-\alpha(s)\|}{|t-s|} : t, s \in [u_n,v_n], \, |t-s|=x \right\}, \quad x \in (0,v_n-u_n]. 
\]
By Lemma \ref{lem:C1}, there exists a measurable (Cantor) set $C_n \subseteq [u_n,v_n]$ satisfying that $|C_n|>0$ and $|C_n \cap I|/|I| < g_n (|I|)$
for every nontrivial interval $I \subseteq [u_n,v_n]$. Write $\Gamma_n := \{\alpha(t) : t \in [u_n, v_n]\}$ and define $F_n \in \lip ( \Gamma_n \cup (\Gamma \setminus B(q_n, r_n/5) ))$ by 
\[
F_n (\alpha(t)) = \int_{u_n}^t \chi_{C_n} (s) - \frac{1}{8} \chi_{[u_n,v_n]\setminus C_n} (s)  \, ds 
\]
for every $t \in [u_n, v_n]$, and $F_n (p) =0$ for $p \not\in B(q_n, r_n/5)$. Let $t_n \in C_n$ be a Lebesgue density point of $C_n$ and $s_n >0$ small enough so that $p_n:=\alpha (t_n)$ satisfies that $B(p_n, s_n) \subseteq \Gamma_n$ (see Figure \ref{fig:d}).

Now, define $H_n \in \lip( \Gamma_n \cup (\Gamma \setminus B(q_n, r_n/5) ) )$ 
by $H_n(\alpha(t)) = F_n (\alpha(t))-F_n(p_n)$ for each $t\in[u_n,v_n]$ and $H_n (p)= 0$ for $p \not\in B(q_n, r_n/5)$.

\textbf{Claim A}: $\|H_n\| =1$ for every $n \in \mathbb{N}$. 

Note from the proof of \cite[Theorem 1.2]{Chiclana22} that $\| F_n \vert_{\Gamma_n} \| =1$. Therefore, $\|H_n\| \geq \|H_n \vert_{\Gamma_n}\| = \|F_n \vert_{\Gamma_n}\| =1$. To prove the claim, it is then enough to consider $|S(F_n, p,q)|$ for $p \in \Gamma_n$ and $q \in \Gamma\setminus B(q_n, r_n/5)$. In this case, letting $p=\alpha(t)$ for some $t \in [u_n,v_n]$, 
\begin{equation}\label{eq:4/7}
|S(F_n, p,q)| = \frac{|H_n(\alpha(t))|}{\| \alpha(t) - q \| } = \frac{|F_n (p) - F_n (p_n)|}{\|p-q\|} \leq  \frac{ \frac{r_n}{20} }{\frac{r_n}{5}-\frac{r_n}{40}} =\frac{2}{7},
\end{equation}
where the last inequality holds since $p, p_n \in B(q_n, r_n/40)$ and $q \not\in B(q_n, r_n/5)$. So, we conclude that $\|H_n\|=1$. 

\begin{figure}[H]
\centering
\includegraphics[width=11cm]{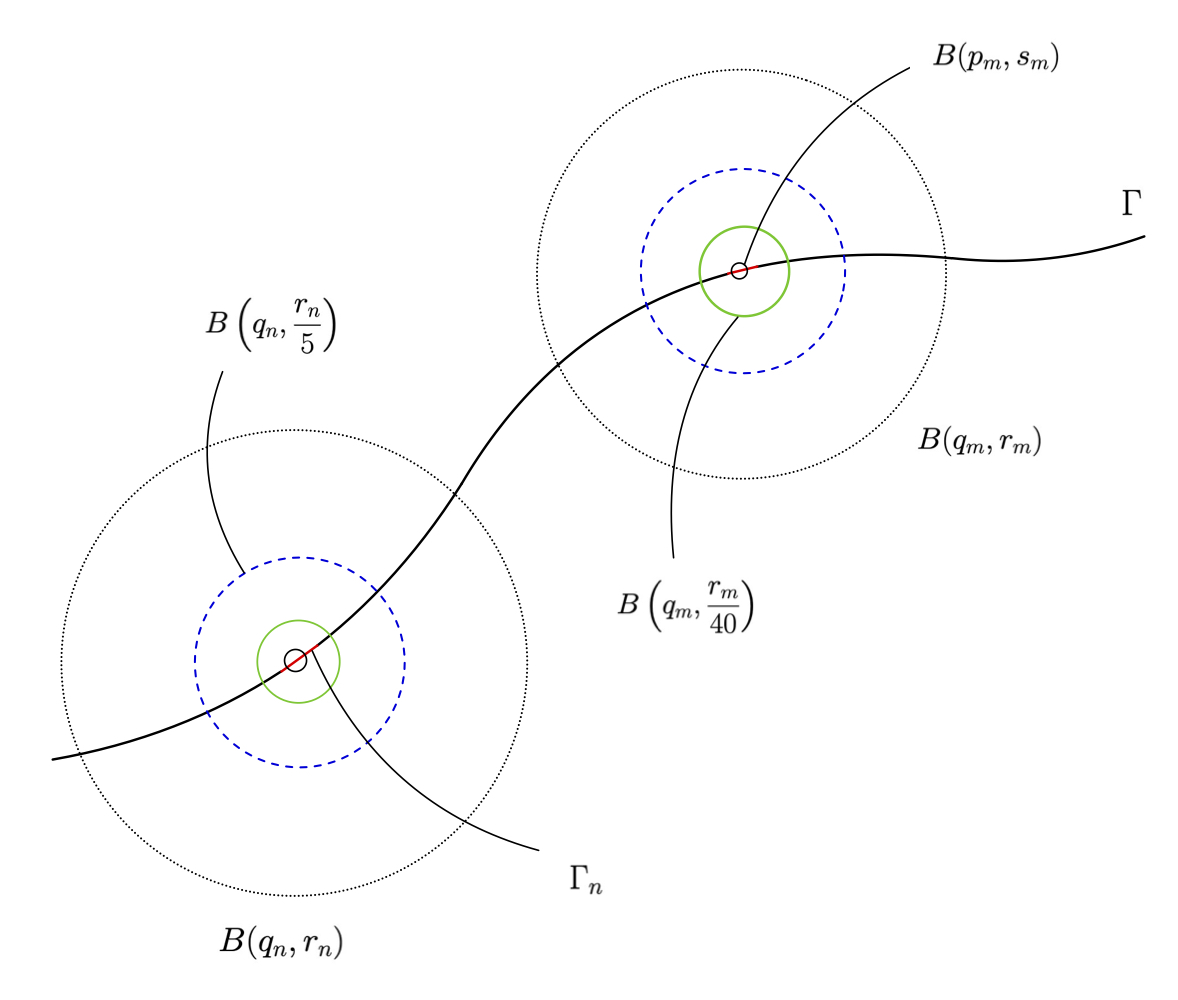}
\caption{The curves and balls}
\label{fig:d}
\end{figure}

\textbf{Claim B}: $H_n \not\in \overline{\sna (\Gamma_n \cup (\Gamma \setminus B(q_n, r_n/5)) )}$ for each $n \in \mathbb{N}$.

Let us first observe the following:
\begin{equation}\label{dagger}
\text{If $G\in \lip (\Gamma_n)$ satisfies that $\|G- F_n \vert_{\Gamma_n}\| < \frac{1}{17}$, then $G \not\in \overline{\sna(\Gamma_n)}$}. \tag{$\dagger$}
\end{equation} 
As a matter of fact, given such a Lipschitz function $G$, consider $\widetilde{G} := \frac{1}{\|\Psi(G)'\|_\infty } G$, where $\Psi: \lip (\Gamma_n) \rightarrow \lip ([u_n,v_n])$ is the isomorphism given by $\Psi(f)(t)=f(\alpha(t))$ for $f \in \lip (\Gamma_n)$ and $t \in [u_n,v_n]$. Note that $\Psi (\widetilde{G})' \in L_\infty ([u_n,v_n])$ has norm one and 
\[
\| \widetilde{G}- F_n \vert_{\Gamma_n}\| < \frac{1}{17}+\|G\| \left| 1 -\frac{1}{\|\Psi(G)'\|_\infty} \right| < \frac{1}{17} + \left(1+\frac{1}{17}\right) \left( \frac{ \frac{1}{17}}{\frac{16}{17}} \right) = \frac{1}{8}.
\]
Thus, $G \not\in \overline{\sna(\Gamma_n)}$ due to the argument in the proof of \cite[Theorem 1.2]{Chiclana22}, where it is proved that if $L \in \lip(\Gamma_n)$ satisfies that $\|\Psi(L)' \|_\infty = 1$ and $\|L-F_n \vert_{\Gamma_n}\| < \frac{1}{8}$, then $L \not\in \overline{\sna (\Gamma_n)}$. 

Back to the Claim B, assume to the contrary that $H_n \in \overline{\sna (\Gamma_n \cup (\Gamma \setminus B(q_n, r_n/5)) )}$. Then there is $L \in \sna (\Gamma_n \cup (\Gamma \setminus B(q_n, r_n/5)) )$ such that $\| L-H_n \| < 1/24$. Suppose that $L$ strongly attains its norm at $m_{\alpha(t), \alpha(s)}$ where $t<s$ are in $I_n \cup [u_n,v_n] \cup J_n$, where $I_n, J_n \subseteq J_0$ are disjoint intervals so that $\alpha(I_n \cup J_n) = \Gamma \setminus B(q_n, r_n/5)$. If $t, s \in [u_n,v_n]$, this would mean that $L\vert_{\Gamma_n}$ strongly attains its norm and $\|L\vert_{\Gamma_n}-F_n \vert_{\Gamma_n}\|  \leq \|L-H_n\| < \frac{1}{24}$. This contradicts to the observation \eqref{dagger}.
Thus, $t$ or $s$ must belong to $I_n \cup J_n$. If $t, s \in I_n \cup J_n$, then 
\[
\frac{1}{24} > \| L - H_n \| \geq \|L\|  > \|H_n\| - \frac{1}{24} = \frac{23}{24}.
\]
which is a contradiction. If $t \in I_n \cup J_n$ and $s \in [u_n,v_n]$, then by \eqref{eq:4/7}, we have 
\[
\frac{1}{24} > \|L-H_n\| \geq \|L\| - |S(H_n, \alpha(t),\alpha(s))| > \left(\|H_n\|-\frac{1}{24}\right) - \frac{2}{7}
\]
which is again a contradiction. The case $s \in I_n \cup J_n$ and $t \in [u_n,v_n]$ can be checked similarly; so we derive a contradiction in any case. Summarizing, if $\|L-H_n\| < \frac{1}{24}$, then $L \not\in \overline{\sna (\Gamma_n \cup (\Gamma \setminus B(q_n, r_n/5)) )}$. 

Next, as in the proof of \cite[Lemma 3.2]{Chiclana22}, we will pick $0<\eps_n<\frac{12\delta_n}{5s_n}$, where $\delta_n$ is a number satisfying that for every $g_0 \in \lip(\Gamma_n \cup (\Gamma \setminus B(q_n, r_n/5)) )$ with $\|g_0 - H_n \| < \delta_n$, we have $g_0 \not\in \overline{\sna (\Gamma_n \cup (\Gamma \setminus B(q_n, r_n/5)) )}$. As we just observed above, $\delta_n$ can be chosen $\frac{1}{24}$ for every $n \in \mathbb{N}$. For our purpose, set 
\[
\eps_n = \frac{12}{10 s_n} \frac{1}{24} = \frac{1}{20 s_n} \, \text{ for all } n \in \mathbb{N}.
\]
Define $\phi_n : \Gamma \rightarrow \mathbb{R}$ by 
$$\phi_n (p) = 
\begin{cases}
1,\quad &\text{if } p \in B\left(p_n, \frac{s_n}{4}\right), \\ 
1-\eps_n \left(d(p,p_n)-  \frac{s_n}{4}\right), \quad &\text{if } p \in B\left( p_n, \frac{s_n}{3}\right) \setminus B\left( p_n , \frac{s_n}{4} \right), \\ 
1-\frac{\eps_n s_n}{12}, \quad &\text{if } p \not\in B\left(p_n, \frac{s_n}{3}\right).
\end{cases}$$
Let us consider a McShane extension $\widetilde{H}_n :  \Gamma\rightarrow [-\frac{r_n}{20}, \frac{r_n}{20}]$ of $H_n$ (see, e.g., \cite[Corollary 2.5]{Matouskova2000}). 
Notice from the calculations in the proof of \cite[Lemma 3.2]{Chiclana22} that $\|\widetilde{H}_n - \widetilde{H}_n \phi_n \| < \delta_n = \frac{1}{24}$. 
Since $\widetilde{H}_n = \widetilde{H}_n \phi_n$ on $B(p_n, \frac{s_n}{4})$ (and as $p_n$ is a norming point for $H_n$), we have $\|\widetilde{H}_n \phi_n\| \geq 1$. For simplicity, put $G_n : = \frac{1}{\|\widetilde{H}_n \phi_n\|} \widetilde{H}_n \phi_n$ for every $n \in \mathbb{N}$. 

Let $\lambda =(\lambda_n)_n \in \ell_\infty$ and $G^{(\lambda)} := \sum_{n=1}^\infty \lambda_n G_n$, where the sum is given by the pointwise limit (noting that the support of $G_n$ is contained in $B(q_n, r_n/5)$ and the balls $B(q_n, r_n/5)$ are disjoint).

\textbf{Claim C}: $\|G^{(\lambda)}\| = \|\lambda\|$.

As the supports of $G_n$'s are disjoint, it is clear that $\|G^{(\lambda)}\|\geq \|\lambda\|$. Let $p \in B(q_n, \frac{r_n}{5})$ and $q \in B(q_m, \frac{r_m}{5})$ for $n \neq m$. Then 
\begin{align}\label{eq:12,20}
&|S(G^{(\lambda)}, p,q)| \nonumber \\
&= \left| \frac{ \frac{\lambda_n}{\|\widetilde{H}_n \phi_n \|} \widetilde{H}_n \phi_n (p) - \frac{\lambda_m}{\|\widetilde{H}_m \phi_m \|} \widetilde{H}_m \phi_m (q) } {\|p-q\|} \right| \nonumber \\ 
&\leq  \frac{|\lambda_n|}{\|\widetilde{H}_n \phi_n \|} \|\widetilde{H}_n \phi_n - \widetilde{H}_n \| + \left| \frac{ \frac{\lambda_n}{\|\widetilde{H}_n \phi_n \|} \widetilde{H}_n (p) - \frac{\lambda_m}{\|\widetilde{H}_m \phi_m \|} \widetilde{H}_m (q) }{\|p-q\|} \right| +  \frac{| \lambda_m| }{\|\widetilde{H}_m \phi_m \|} \|\widetilde{H}_m \phi_m - \widetilde{H}_m \| \nonumber \\ 
&\leq \frac{1}{24} (|\lambda_n| + |\lambda_m|) + \frac{|\lambda_n| \frac{r_n}{20} + |\lambda_m| \frac{r_n}{20}}{\|p-q\|} \nonumber \\ 
&< \|\lambda\| \left( \frac{1}{12} + \frac{ \frac{r_n}{20} + \frac{r_m}{20}}{r_n + r_m} \right) = \|\lambda\| \left( \frac{1}{12} + \frac{1}{20} \right). 
\end{align} 
If $p \in B(p_n, \frac{r_n}{5})$ and $q \not\in \cup_{m=1}^\infty B(p_m, r_m/5)$, then $|S(G^{(\lambda)},p,q)| = |\lambda_n S(G_n, p,q)| \leq \|\lambda\|$. This proves that $\|G^{(\lambda)}\| = \|\lambda\|$.

\textbf{Claim D}: $G^{(\lambda)} \not\in \overline{\sna(\Gamma)}$. 

Assume to the contrary that there exists $L \in \sna(\Gamma)$ such that $\|L-G^{(\lambda)}\| < \beta$, where 
\[
0< \frac{ \beta} {\|\lambda\|-\beta} < \frac{24}{25} \left( \frac{1}{17} - \frac{1}{24} \right) \,\,  \text {and } \,\,  0 < \beta < \frac{\|\lambda\|}{480}. 
\]
Suppose that $|L(m_{x,y})| =\|L\| = \|G^{(\lambda)}\|$. If $x,y \not\in \cup_{m=1}^\infty B(q_m, \frac{r_m}{5})$, then $\|L-G^{(\lambda)}\| \geq |L(m_{x,y})| = \|\lambda\|$, which is a contradiction. Thus, by relabelling $x$ and $y$ if necessary, we assume that $x \in B(q_n, \frac{r_n}{5})$ for some $n \in \mathbb{N}$. 

\underline{Case 1}: $y \in B(q_m, \frac{r_m}{5})$ for some $m \in \mathbb{N}$ with $m \neq n$. Then as in \eqref{eq:12,20} we observe that 
\[
|S(G^{(\lambda)}, x, y)| \leq \|\lambda\| \left( \frac{1}{12} + \frac{1}{20}\right),
\]
which contradicts that $|S(G^{(\lambda)}, x, y)| > \|\lambda\|-\beta$. 

\underline{Case 2}: $y \not\in \cup_{m=1}^\infty B(q_m, \frac{r_m}{5})$. We consider two subcases: 

\underline{Subcase 2-1}: If $x \in B(p_n, \frac{s_n}{3})$, then noting that $B(p_n, \frac{s_n}{3}) \subseteq B(q_n, \frac{r_n}{5})$, we have 
\[
|S(G^{(\lambda)}, x, y)| = |\lambda_n| \frac{|G_n(x)|}{\|x-y\|} < \|\lambda\| \frac{ \frac{r_n}{20} }{\frac{r_n}{5}-\frac{r_n}{40}} = \frac{2}{7} \|\lambda\|, 
\]
which is a contradiction.

\underline{Subcase 2-2}: If $x \not\in B(p_n, \frac{s_n}{3})$. As it is clear that $y \not\in B(p_n, \frac{s_n}{3})$, we observe from the definition of $\phi_n$ that 
\begin{align*}\label{eq:240}
\beta > \|G^{(\lambda)}-L\| &\geq \|\lambda\| - |\lambda_n| \frac{|G_n(x)|}{\|x-y\|} \nonumber \\
&\geq \| \lambda \| - \|\lambda\| \left( 1- \frac{\eps_n s_n}{12} \right) \| \widetilde{H}_n \| \frac{1}{\|\widetilde{H}_n \phi_n \|} \nonumber \geq \|\lambda\| \left( 1 - \left( 1 - \frac{1}{240}\right) \right) = \frac{\|\lambda\|}{240},
\end{align*} 
which is a contradiction again.

Finally, the only remain case is: 

\underline{Case 3}: $y \in B(q_n, \frac{r_n}{5})$. 

\underline{Subcase 3-1}: If $x, y \in B(p_n, s_n)$, then this would imply that $L \vert_{\Gamma_n}$ strongly attains its norm. Note that 
\[
\beta > \| L \vert_{\Gamma_n} - G^{(\lambda)} \vert_{\Gamma_n} \| = \left \| L \vert_{\Gamma_n} - \frac{\lambda_n}{\| \widetilde{H}_n\phi_n\| } \widetilde{H}_n \phi_n \vert_{|\Gamma_n} \right \|. 
\]
This implies that $|\lambda_n| > \|\lambda\| - \beta$ and, putting $G_0 := \frac{\| \widetilde{H}_n\phi_n\|}{\lambda_n} L \vert_{\Gamma_n}$ ($\in \sna(\Gamma_n)$), we have that 
\begin{align*}
\|H_n \vert_{\Gamma_n} - G_0 \| &\leq \|\widetilde{H}_n - \widetilde{H}_n \phi_n \| + \|\widetilde{H}_n \phi_n - G_0 \| \\
&< \frac{1}{24} +  \frac{\| \widetilde{H}_n\phi_n\|}{\lambda_n} \beta < \frac{1}{24} + \frac{1+\frac{1}{24}}{\|\lambda\|-\beta} \beta < \frac{1}{17},
\end{align*}
which contradicts to the above observation \eqref{dagger}. 

\underline{Subcase 3-2}: If $x \in B(p_n, \frac{s_n}{3})$, then by Subcase 3-1, $y$ cannot be in $B(p_n, s_n)$. Then 
\begin{align*}
&|S(G^{(\lambda)}, x,y)| = \frac{| \lambda_n| }{\| \widetilde{H}_n \phi_n\|} \left| \frac{ \widetilde{H}_n (x)\phi_n (x) - \widetilde{H}_n (y) \phi_n (y) }{\|x-y\|} \right| \leq \|\lambda\| \left| (I) \right|,
\end{align*} 
where 
\[
(I) =  \frac{ \widetilde{H}_n (x)\phi_n (x) - \left(1-\frac{\eps_n s_n}{12}\right)\widetilde{H}_n(y) }{\|x-y\|}. 
\] 
If $(I) \geq 0$, then 
\begin{align*}
|(I)| &\leq \frac{ \widetilde{H}_n (x)\phi_n (x) - \left(1-\frac{\eps_n s_n}{12}\right) (\widetilde{H}_n (x)- \|x-y\|) }{\|x-y\|} \\
&= \left( 1-\frac{\eps_n s_n}{12} \right) + \frac{H_n (x) \left[\phi_n (x) - \left( 1- \frac{\eps_n s_n}{12} \right)\right]}{\|x-y\|} < \left( 1-\frac{\eps_n s_n}{12} \right)  + \frac{ \frac{s_n}{3} \left( \frac{\eps_n s_n}{12} \right) }{\frac{2}{3} s_n} = 1 - \frac{1}{480},
\end{align*} 
which implies that $|S(G^{(\lambda)}, x,y)| \leq \|\lambda\| (1- \frac{1}{480})$, a contradiction. 

If $(I) \leq 0$, then 
\begin{align*}
|(I)| &= \frac{ \left(1-\frac{\eps_n s_n}{12}\right) \widetilde{H}_n (y) - \widetilde{H}_n (x)\phi_n (x) }{\|x-y\|} \\ 
&\leq \frac{ \left(1-\frac{\eps_n s_n}{12}\right) (H_n (x) + \|x-y\|) - \widetilde{H}_n (x)\phi_n (x) }{\|x-y\|} \\ 
&= \left( 1-\frac{\eps_n s_n}{12} \right) + H_n (x) \left( \frac{ \left(1 - \frac{\eps_n s_n}{12} \right) - \phi_n (x) }{\|x-y\| } \right) \\
&\leq \left(1-\frac{\varepsilon_n s_n}{12}\right) + \frac{|H_n(x)|}{\|x-y\|} |\phi(x)-\phi(y)| \leq \left(1-\frac{\varepsilon_n s_n}{12}\right) + \frac{\frac{s_n}{3}}{\frac{2 s_n}{3}} \frac{\varepsilon_n s_n}{12} = 1 - \frac{\varepsilon_n s_n}{24};
\end{align*} 
so it is a contradiction again.

\underline{Subcase 3-3}: Suppose that $x \not\in B(p_n, \frac{s_n}{3})$. If $y \in B(p_n, \frac{s_n}{3})$, then $x$ cannot be in $B(p_n, s_n)$ by Subcase 3-1. Then we argue as in Subcase 3-2 to derive that 
\[
|S(G^{(\lambda)}, y, x) | < \|\lambda\| \left( 1- \frac{1}{480} \right),
\]
which is a contradiction. Thus, $y \not\in B(p_n, \frac{s_n}{3})$. In this case, we compute as in Subcase 2-2 that 
\[
\beta > \|G^{(\lambda)} - L \| \geq \frac{\| \lambda \|}{240},
\]
so a contradiction.

Summarizing, if $L \in \lip(\Gamma)$ satisfies that $\|L-G^{(\lambda)}\| < \beta$, then $L \not\in \overline{\sna(\Gamma)}$. 
\end{proof}

We finish this section by presenting the spaceability results concerning the set of Lipschitz functions which attain their pointwise norm but are not strongly norm-attaining. Recall that $f \in \Lip(M,Y)$ is said to \textit{attain its pointwise norm} if there exists $p \in M$ such that
$$\sup_{q \in M \setminus \{p\}} \| S(f,p,q)\| = \|f\|.$$
In this case, we write $f \in \PNA(M,Y)$ (for simplicity, $f \in \pna(M)$ when $Y=\mathbb{R}$). The concept of pointwise norm attainment of Lipschitz functions is examined in \cite{CJLR2023} and it has been noted that given an infinite metric space $M$, there is always a subset $M_0\subset M$ such that the set of Lipschitz functions on $M_0$ attaining their pointwise norm contains an isometric copy of $c_0$ (\cite[Theorem 4.11]{CJLR2023}). Given a metric space $M$, let $M'$ denote the set of its accumulation points.

\begin{theorem}\label{prop:nosna-c0-Gamma}
Let $M$ be an infinite pointed metric space such that $M'$ has infinite density character $\Gamma$. Then, we have the following:
\begin{enumerate}
\itemsep0.25em
\item[\textup{(a)}] $(\pna(M) \setminus \sna(M)) \cup \{0\}$ isometrically contains $c_0(\Gamma)$;
\item[\textup{(b)}] If, in addition, $\Gamma \geq 2^{\aleph_0}$, then $(\pna(M) \setminus \sna(M))\cup \{0\}$ isometrically contains $\ell_1$.
\end{enumerate} 
\end{theorem}

\begin{proof}
Take, by using \cite[Proposition 5.1]{DMQR23}, a discrete metric subspace $\{p_\gamma\}_{\gamma\in\Gamma}\subset M' \setminus \{0\}$ such that $p_\alpha\neq p_\beta$ whenever $(\alpha, \beta)\in\widetilde{\Gamma}$.
For each $\gamma \in \Gamma$, choose $q_\gamma \in M$ so that 
\[
d(p_\gamma,q_\gamma) < \frac{1}{2} \inf\{ d(p_{\gamma},p_{\alpha}): \alpha \in \Gamma \setminus \{ \gamma\} \}. 
\]
The choice of $\{q_\gamma\}_{\gamma \in \Gamma}$ yields that $d(p_\alpha,p_\beta)> d(p_\alpha, q_\alpha) + d(p_\beta, q_\beta)$ for every $(\alpha,\beta)\in\widetilde{\Gamma}$. 
For each $\gamma \in \Gamma$, define the function $\phi_\gamma\in\lip(M)$ as
$$
\phi_{\gamma} (p) := \begin{cases}
\frac{2d(p_{\gamma},q_{\gamma})}{\pi}\left(1-\sin\left( \frac{\pi d(p_{\gamma}, p)}{2d(p_{\gamma}, q_{\gamma})} \right)\right),\ &\text{if $d(p_{\gamma}, p)\leq d(p_{\gamma}, q_{\gamma})$},\\
0,\ &\text{if $d(p_{\gamma}, p)\geq d(p_{\gamma}, q_{\gamma}).$}
\end{cases}
$$

(a): 
Given $a=\{a_\gamma:\, \gamma\in\Gamma\}\in c_0(\Gamma)\setminus\{0\}$, let $f^{(a)}:=\sum_{\gamma\in\Gamma} a_\gamma \phi_\gamma$.
We claim that $f^{(a)} \in \pna(M) \setminus \sna(M)$ with $\|f^{(a)}\|=\|a\|_\infty$. 
Since $\supp(\phi_\alpha)\cap \supp(\phi_\beta)=\emptyset$ for all $(\alpha, \beta)\in\widetilde{\Gamma}$, we clearly have 
\[
\|f^{(a)}\|\geq \max\{|a_\gamma| \|\phi_\gamma\|:\, \gamma\in \Gamma\}=\|a\|_\infty.
\]
Next, we show that $\|f^{(a)}\|\leq \|a\|_\infty$ and $f^{(a)} \in \pna(M)\setminus \sna(M)$. Let $(p,q)\in\widetilde{M}$ be given. 
\begin{itemize}
\itemsep0.25em
\item If $p, q \not\in \cup_{\gamma \in \Gamma} B(p_\gamma, d(p_\gamma, q_\gamma))$, then $S(f^{(a)}, p,q)=0$. 
\item Suppose that $p \in B(p_\alpha, d(p_\alpha,q_\alpha))$ for some $\alpha \in \Gamma$ while $q \not\in \cup_{\gamma \in \Gamma} B(p_\gamma, d(p_\gamma, q_\gamma))$. Since the function $\phi_\gamma$ is convex throughout every ray of $B(p_\gamma,\, d(p_\gamma, q_\gamma))$, we have that for each $\gamma \in \Gamma$, 
\begin{equation}\label{eq:convexity_phi}
\phi_\gamma(x) \leq \max\left\{0, \frac{2}{\pi}\left( d(p_{\gamma}, q_{\gamma})-d(p_{\gamma}, x) \right)\right\} =: T_\gamma (x) \, \text{ for every } x \in M.
\end{equation}
Thus, 
\begin{align*}
|S( f^{(a)}, p , q) | = \frac{|a_\alpha \phi_\alpha(p)|}{d(p,q)} &\leq |a_\alpha| \frac{T_\alpha (p)}{d(p,q)}\\
&\leq \frac{2 |a_n|}{\pi d(p,q)} (d(p_\alpha,q_\alpha)-d(p_\alpha,p)) \\ 
&< \frac{2 |a_n|}{\pi d(p,q)} (d(p_\alpha,q)-d(p_\alpha,p)) \leq  \frac{2 }{\pi}\|a\|_\infty. 
\end{align*}

\item Suppose that there exists $\alpha \in {\Gamma}$ such that $p, q \in B(p_\alpha,\, d(p_\alpha, q_\alpha))$. Without loss of generality, assume that $ d(p_\alpha,q)\leq d(p_\alpha,p)$. 
Then,
\begin{equation*}\label{eq:spike-1}
|S(f^{(a)}, p, q)|=|a_\alpha| |S(\phi_\alpha, p, q)|=|a_\alpha|\frac{2 d(p_\alpha, q_\alpha)}{\pi}\frac{\sin\left( \frac{\pi d(p_\alpha, p)}{2 d(p_\alpha, q_\alpha)} \right) - \sin\left( \frac{\pi d(p_\alpha, q)}{2 d(p_\alpha, q_\alpha)} \right)}{ d(p,q)}.
\end{equation*}
Note that $|\sin(r)-\sin(s)| \leq |r-s|$ for $r,s\in \mathbb{R}$ and the inequality is strict unless $r=s$. Therefore, 
$$
|S(\phi_\alpha, p, q)| < \frac{2 d(p_\alpha, q_\alpha)}{\pi}\frac{\frac{\pi d(p_\alpha, p)}{2 d(p_\alpha, q_\alpha)}-\frac{\pi d(p_\alpha, q)}{2 d(p_\alpha, q_\alpha)}}{ d(p,q)}=\frac{ d(p_\alpha, p) - d(p_\alpha, q)}{ d(p, q)}\leq 1.
$$

Moreover, choose a sequence $\{p_n\}_n\subset B(p_\alpha,\, d(p_\alpha, q_\alpha))$ such that $d(p_\alpha, p_n)$ converges decreasingly to $0$ as $n$ tends to infinity. 
Then 
\begin{equation}\label{eq:phi_pna}
\lim_{n\to\infty} |S(\phi_\alpha, p_{\alpha}, p_n)| = \lim_{n\to\infty} \frac{\sin\left( \frac{\pi d(p_\alpha, p_n)}{2d(p_\alpha, q_\alpha)} \right)}{\frac{\pi d(p_\alpha, p_n)}{2d(p_\alpha, q_\alpha)}} = 1.
\end{equation}

\item Suppose that there exist some $(\alpha, \beta) \in \widetilde{\Gamma}$ such that $p\in B(p_\alpha,\, d(p_\alpha, q_\alpha))$ and $q\in B(p_\beta,\, d(p_\beta, q_\beta))$ with $a_\alpha \neq 0$ and $a_\beta \neq 0$. We assume that $\alpha \neq \beta$. 
In this case,
\[
|S(f^{(a)}, p, q)|  = \frac{|a_\alpha \phi_\alpha (p) - a_\beta \phi_\beta (q)|}{d(p,q)}. 
\]
Consider $T:= a_\alpha T_\alpha + a_\beta T_\beta \in \lip(M)$.
Notice  that  
\begin{align*}
|S(T,p,q)| &= \frac{|a_\alpha| T_\alpha (p) + |a_\beta| T_\beta (q)}{d(p,q)} \\
&\leq \frac{2\|a\|_\infty}{\pi d(p,q)} (d(p_\alpha,q_\alpha)-d(p_\alpha,p)+d(p_\beta,q_\beta)-d(p_\beta,q)) \\ 
&\leq \frac{2\|a\|_\infty}{\pi d(p,q)} (d(p_\alpha, p_\beta) - d(p_\alpha, p) - d(p_\beta, q)) \leq \frac{2}{\pi}\|a\|_\infty. 
\end{align*}
As a consequence, we have 
\[
|S(f^{(a)}, p, q)| \leq |S(T, p, q)| \leq \frac{2}{\pi}\|a\|_\infty < \|a\|_\infty.
\]
\end{itemize}

Finally, let $\gamma_0\in\Gamma$ be such that $|a_{\gamma_0}|=\|a\|_\infty$, and note from \eqref{eq:phi_pna} that
$$
\sup_{r\in M\setminus\{p_{\gamma_0}\}} |S(f^{(a)}, p_{\gamma_0}, r)| = \sup_{r\in B(p_{\gamma_0}, d(p_{\gamma_0},q_{\gamma_0}))} |a_{\gamma_0}| |S(\phi_{\gamma_0}, p_{\gamma_0}, r)| = \|a\|_\infty.
$$
This concludes the proof.

(b): Suppose that $\Gamma \geq 2^{\aleph_0}$. Let $S = \{ -1, 1 \}^\mathbb{N} = \{ s^{(\gamma)} \}_{\gamma \in \Gamma_1}$, where $\Gamma_1\subset \Gamma$ is an index set with cardinality $|\Gamma_1|=2^{\aleph_0}$. 
For each $n \in \mathbb{N}$, consider $e_n \in \Lip(M)$ given by 
\[
e_n (p) = \begin{cases}
s^{(\gamma)} (n) \phi_\gamma (p),\ &\text{if $d(p_{\gamma}, p)\leq d(p_{\gamma}, q_{\gamma})$},\\
0,\ &\text{if $d(p_{\gamma}, p)\geq d(p_{\gamma}, q_{\gamma}),$}
\end{cases}
\]
where $s^{(\gamma)}= (s^{(\gamma)} (n))_{n=1}^\infty$. Note that $\|e_n\|=1$ since $\|\phi_\gamma\|=1$ for every $\gamma \in \Gamma_1$. 

For each $a= (a_n)_n \in \ell_1$, let $g^{(a)} := \sum_{n=1}^\infty a_n e_n$. Then it is clear that $\|g^{(a)} \| \leq \|a\|_1$. To see that $\|g^{(a)}\| \geq \|a\|_1$, consider $b \in S$ given by 
\[
b_n = 1 \text{ for every } n \in \mathbb{N} \text{ with } a_n = 0, \, \text{ and } \, b_n = \sign (a_n) \text{ for every } n \in \mathbb{N} \text{ with } a_n \neq 0.
\]
Pick $\gamma_0 \in \Gamma_1$ so that $b = s^{(\gamma_0)}$. Then 
\begin{align*}
\sup_{q \neq p_{\gamma_0}} |S(g^{(a)}, p_{\gamma_0}, q)| &\geq \sup_{q \in B(p_{\gamma_0}, d(p_{\gamma_0}, q_{\gamma_0}))} |S(g^{(a)},p_{\gamma_0}, q)| \\
&=  \sup_{q \in B(p_{\gamma_0}, d(p_{\gamma_0}, q_{\gamma_0}))} \left(\sum_{n=1}^\infty |a_n| \right) |S(\phi_{\gamma_0},p_{\gamma_0}, q)| = \|a\|_1.
\end{align*}

This shows that $\|g^{(a)}\| = \|a\|_1$ and $g^{(a)}$ attains its pointwise norm at $p_{\gamma_0}$, and note that $g^{(a)}$ does not strongly attain its norm, arguing as in case (a).
\end{proof}

\section{Linearity of strongly norm-attaining Lipschitz functions}\label{section:linearity}

In this section, we study whether the sets $\sna(M)$ and $(\lip(M)\setminus \sna(M))\cup\{0\}$ can be linear spaces themselves for some suitable metric space $M$. This leads to solve the problem \ref{q2} in the negative. Recall first the properties \ref{a1} and \ref{a2} that are essential in the proof of Theorem \ref{Main-Theorem-c0-PNA}. By slightly modifying \ref{a2}, we obtain the following.  

\begin{proposition}\label{prop:nosna-c0-and-non-linearity}
Let $M$ be an infinite pointed metric space such that $\sna(M)$ contains an isometric copy of $c_0$ with basis $\{f_n\}_n$ satisfying the following properties:
\begin{enumerate}
\itemsep0.25em
\item[\textup{(a1)}] if $n\neq m$, then $\supp(f_n) \cap \supp (f_m) = \emptyset$;
\item[\textup{(a2')}] given $a=(a_n)_n \in\ell_\infty$, if $f^{(a)}$ is the pointwise limit of $\sum_{n} a_n f_n$, then $\|f^{(a)}\|=\|a\|_\infty$, and moreover, $f^{(a)}$ strongly attains its norm if and only if there exists $n_0\in\bbn$ such that $|a_{n_0}|=\|a\|_\infty$.
\end{enumerate}
Then the sets $\sna(M)$ and $(\lip(M)\setminus \sna(M))\cup\{0\}$ are not linear spaces. 
\end{proposition}

\begin{proof}
Pick any $z=(z_n)_{n=1}^\infty \in S_{\ell_\infty}$ so that $z_1 \neq 0$ and $|z_n| < 1$ for every $n \in \mathbb{N}$. Define $w=(w_n)_{n=1}^\infty \in S_{\ell_\infty}$ by $w_1 = 0$ and $w_n = z_n$ for every $n>1$. Note that $f_1$ and $f_1+f^{(w)}$ are both in $\sna(M)$, while $f^{(w)}$ is not, so $\sna(M)$ cannot be a linear space. Similarly, note that $f^{(z_1 e_1)}=f^{(z)}-f^{(w)}\in\sna(M)$, while $f^{(z)},f^{(w)}\notin \sna(M)$. This shows that $(\lip(M)\setminus \sna(M))\cup\{0\}$ is not a linear space.\qedhere
\end{proof}

The following lemma will be used later. 

\begin{lemma}\label{lemma:pna-minus-sna-non-empty-nosna}
Let $M$ be an infinite pointed metric space such that $\pna(M)\setminus \sna(M)\neq \emptyset$. Then $(\lip(M)\setminus \sna(M))\cup\{0\}$ is not a linear space.
\end{lemma}

\begin{proof}
Let $f\in \pna(M)\setminus \sna(M)\neq \emptyset$ be a norm one Lipschitz function that attains its pointwise norm at some $p_0\in M$. Without loss of generality, we assume that $p_0$ is the distinguished point $0$. 
Note that, up to multiplying $f$ by $-1$ if needed, there exists $\{p_n\}_{n=1}^{\infty}\subset M$ such that $S(f, p_0, p_n)$ converges to $1$. Consider $g\in\lip(M)$ given by $g(p):=d(p_0, p)$ for every $p\in M$, and note that $\|g\|=1$ and $g\in \sna(M)$. Since $S(f+g, p_0, p_n)$ converges to $2$, we have that $\|f+g\|=2$. However, $f+g \not\in\sna(M)$ since if $f+g$ strongly attains its norm at some pair $(p,q)$, then so does $f$, which is a contradiction. Summarizing, $f+g$ and $f$ are not in $\sna(M)$ while $g = (f+g)-f \in \sna(M)$.
\end{proof}

\begin{theorem}\label{thm:sna_not_linear}
For any infinite metric space $M$, we have that the sets $\sna(M)$ and $(\lip(M)\setminus \sna(M))\cup\{0\}$ are not linear spaces. 
\end{theorem}

\begin{proof}
If $M$ is not uniformly discrete, then the proof of \cite[Theorem 4.2]{DMQR23} provides $c_0 \stackrel{1}{=} \overline{\spann}  \{f_n\} \subseteq \sna(M)$ such that $\{f_n\}$ satisfies the properties in Proposition \ref{prop:nosna-c0-and-non-linearity}. Therefore, the sets $\sna(M)$ and $(\lip(M)\setminus \sna(M))\cup\{0\}$ are not linear spaces. Thus, we may assume $M$ to be uniformly discrete. Denote $R:=\inf\{d(x,y):\, (x,y)\in\widetilde{M}\}$.

\textbf{Claim A}: $\sna(M)$ is not a linear space. 

Let $g\in\lip(M)\setminus \sna(M)$ be such that $\|g\|=1$ (such function exists as $\mathcal{F}(M)$ is not reflexive). Let $(x_0, y_0)\in\widetilde{M}$ be such that $d(x_0, y_0)<\frac{3R}{2}$. Define $h_1:M\rightarrow \bbr$ as $h_1(x_0)=9R$, $h_1(y_0)=-9R$, and $h_1(p)=0$ for all $p\in M\setminus\{x_0, y_0\}$, and let $h=h-h(0)\in\lip(M)$.

Let $(p,q)\in\widetilde{M}$. We clearly have that
$$\|h\|\geq |S(h, x_0, y_0)|=\frac{18R}{d(x_0, y_0)}>\frac{18R}{\frac{3}{2}R}=12.$$
On the other hand, if $\{p,q\}\neq\{x_0, y_0\}$, we have
$$|S(h, p, q)|\leq \frac{9R}{d(p,q)}\leq \frac{9R}{R}=9.$$
So $\|h\|>12$, and $h$ strongly attains its norm (only) at $(x_0,y_0)$.

Consider now the function $f=g+h$, and let $(p,q)\in\widetilde{M}$. We have the following.
\begin{itemize}
\itemsep0.25em
\item If $\{p,q\}=\{x_0, y_0\}$, then $|S(f,p,q)|\geq |S(h,x_0, y_0)| - |S(g,x_0,y_0)|>12-1=11$.
\item If $\{p,q\}\neq \{x_0, y_0\}$, then $|S(f, p, q)|\leq |S(h, p, q)| + |S(g, p, q)| < 9+1=10$.
\end{itemize}
Therefore, $\|f\|>11$, and $f$ strongly attains its norm (only) at $(x_0, y_0)$. As a consequence, $\sna(M)$ is not a linear space.

\textbf{Claim B}: $(\lip(M)\setminus \sna(M))\cup\{0\}$ is not a linear space.

For each $p\in M$, let $\neib(p):=\{q\in M:\, d(p,q)=R(p)\}$. We distinguish two cases.

\underline{Case 1}: Suppose that there exists $p_0 \in M$ such that $\neib(p_0 )=\emptyset$.

Let $f:M\rightarrow \bbr$ be given by $f(p_0 )=R(p_0)$ and $f(p)=0$ for all $p\in M\setminus\{p_0\}$. Note that $|S(f,p,q)| < 1$ for every $(p,q) \in \widetilde{M}$ and $|S(f,p_0, q_n)| \rightarrow 1$ whenever $\{q_n\}_{n=1}^\infty \subseteq M$ is such that $d(p_0, q_n) \rightarrow R(p_0)$. 
Therefore, $\|f-f(0)\|=\|f\|=1$ and $f-f(0) \in \pna(M) \setminus \sna(M)$. By Lemma \ref{lemma:pna-minus-sna-non-empty-nosna}, we conclude that $(\lip(M)\setminus \sna(M))\cup\{0\}$ is not a linear space.

\underline{Case 2}: For every $p\in M$, the set $\neib(p)$ is non-empty.

Recall from the proof of \cite[Theorem 5]{CJ17}, we can find a sequence $\{x_n\}_n \subseteq M\setminus \{0\}$ and $\{r_n\}_n \subseteq [0,\infty)$ satisfying $d(x_n,x_m) \geq r_n + r_m$ for every $n\neq m$, and 
\[
\lim_{n\rightarrow\infty} \frac{r_{2n} + r_{2n+1}}{d(x_{2n}, x_{2n+1})} = 1. 
\]
Let us point out that the proof of \cite[Theorem 5]{CJ17} actually shows that $\{r_n\}_n$ can be chosen so that $r_n \geq \frac{R}{4}$ for all $n\in\bbn$.
Let us consider $M$ as a metric subspace of $\mathcal{F}(M)$. Given disjoint increasing sequences $\{n_\ell^k\}_{\ell=1}^\infty \subseteq \mathbb{N}$, $k \in \mathbb{N}$, put $\widetilde{f}_k = \sum_{\ell=1}^\infty g_{2n_\ell^k} - g_{2n_\ell^k + 1} : \mathcal{F}(M) \rightarrow \mathbb{R}$, where $g_n(\mu) = \max \{r_n - \|\mu- {x_n}\|, 0\}$ for every $\mu \in \mathcal{F}(M)$. For simplicity, set $f_k := \widetilde{f}_k \vert_M \in \lip(M)$ for every $k \in \mathbb{N}$. From the proof of \cite[Theorem 5]{CJ17}, we observe $\{f_n \}_n \subseteq \lip(M)$ satisfies the conditions \ref{a1} and \ref{a2} in Theorem \ref{Main-Theorem-c0-PNA}. 

Consider the subspace $\{ z^{(a)}: a \in \ell_\infty\}$ of $\ell_\infty$ from the proof of Lemma \ref{lemma:special-c0-in-ell-infty}. 
For each $a=(a_n)_n \in \ell_\infty$, let us define $h^{(a)} \in \lip(M)$ to be the pointwise limit of
\[
h^{(a)} := f^{( z^{(a)} )} = \sum_n z^{(a)} (n) f_n. 
\]
Then $\|h^{(a)}\| = \|z^{(a)}\|_\infty = \|a\|_\infty$, and the argument from the proof of Theorem \ref{Main-Theorem-c0-PNA} shows that $h^{(a)}\in (\lip(M)\setminus \sna(M))\cup\{0\}$.

Fix $p_0=x_{2n_1^2}\in \supp(f_2)$. Let us denote by $\{e_n\}_n$ the canonical coordinate vectors in $\ell_\infty$. Then we have $\supp(h^{(e_n)})\cap \{p_0\}=\emptyset$ for all $n >1$. Consider $h^{(e_2)}\in \lip(M)\setminus \sna(M)$ and $f\in\lip(M)$ given by $f(p_0)=\frac{R}{8}$ and $f(p)=0$ for all $p\in M\setminus\{p_0\}$. Since $\neib(p_0)$ is non-empty, it is clear that $\|f\|=\frac{R/8}{R(p_0)}$ and $f$ strongly attains its norm (only) at every pair $(p_0,q)$ with $q\in\neib(p_0)$. Finally, consider the function $g:=h^{(e_2)}+f \in \lip(M)$ and let $(p,q)\in\widetilde{M}$ be given.
\begin{itemize}
\itemsep0.25em
\item If $p,q\notin \{p_0\}\cup \supp(h^{(e_2)})$, then $|S(g, p, q)|=0$.
\item If $p=p_0$ and $q\notin \supp(h^{(e_2)})$, then $|S(g, p, q)|=|S(f,p,q)|\leq \frac{\frac{R}{8}}{R(p)}\leq \frac{1}{8}$.
\item Suppose that $p=p_0$ and $q\in\supp(h^{(e_2)})$. Note first from the construction that 
\[
h^{(e_1)} (p) = \left( \sum_{n=1}^\infty z^{(e_1)} (n) f_n \right) (x_{2n_1^2}) = \left( \sum_{m=1}^\infty (1-2^{-m}) f_{2^m} \right) (x_{2n_1^2}) = \frac{r_{2n_1^2}}{2}.
\]
Therefore, setting $\alpha = \frac{R}{4 r_{2n_1^2}}$, we have $\alpha h^{(e_1)}(p) = \frac{R}{8}$.
From our choice of $\{r_n\}_n$, note that $\alpha \in (0,1]$. Now, 
\begin{align*}
|S(g, p, q)| = \left|\frac{h^{(e_2)}(q) - \frac{R}{8}}{d(p,q)}\right| = |S(\alpha h^{(e_1)} + h^{(e_2)}, p, q)| <  \| h^{ (\alpha e_1 + e_2 )} \|=1. 
\end{align*} 
\item If $p\in \supp(h^{(e_2)})$ and $q\neq p_0$, then $|S(g, p, q)|=|S(h^{(e_2)},p,q)|<1$. 
\end{itemize}
Therefore, $\|g\|=1$ and $g\in \lip(M)\setminus \sna(M)$. Since $g-h^{(e_2)} = f \in\sna(M)$, we conclude that $(\lip(M)\setminus \sna(M))\cup\{0\}$ is not a linear space.
\end{proof}

\section{Spaceability of some non-norm attaining Lipschitz functions}\label{LipNA}

In recent years, there has been growing interest in the study of Lipschitz functions which attain their norm in many different ways, due to the limitations of strong norm-attainment for Lipschitz functions (see \cite{CCGMR19,Choi23,CCM20,Godefroy16,KMS16} for instance). Along this direction, we present some further spaceability results for classes of Lipschitz functions which attain their norm in one way but do not in another.
First, let us recall some notations.

Let $X$ and $Y$ be real Banach spaces. Recall that $f\in\lip(X, Y)$ \textit{attains its norm locally directionally} at a point $\overline{x}$ in the direction $u\in S_X$ toward $z\in \|f\| S_Y$ if there exists $\{(p_n, q_n)\}_{n=1}^{\infty}\subset \widetilde{X}$ such that
$$
S(f,p_n,q_n)\longrightarrow z,\quad \frac{p_n-q_n}{\|p_n-q_n\|}\longrightarrow u\quad \text{and}\quad p_n,q_n\longrightarrow \overline{x}.
$$
The set of Lipschitz function $f\in\lip(X, Y)$ that attain their norm locally directionally is denoted by $\ldira(X, Y)$.

We say that $f\in\lip(X, Y)$ \textit{attains its norm through a derivative} at a point $x\in X$ in the direction $e\in S_X$ if 
$$
f'(x,e):=\lim_{t\to 0}\frac{f(x+te)-f(x)}{t}\in Y\text{ exists and } \quad\|f'(x,e)\|=\|f\|.
$$
The set of Lipschitz function $f\in\lip(X, Y)$ that attain their norm through a derivative is denoted by $\der(X, Y)$.

We will be only considering real-valued Lipschitz functions since analogous results for the vector-valued case can be obtained immediately by fixing a vector. It was shown in \cite{CJLR2023} that for every real Banach space $X$, we have
$$\der(X)\subset \pna(X)\cap \ldira(X),\,\, \pna(X)\setminus \ldira(X)\neq\emptyset,\,\, \text{and}\,\, \ldira(X)\setminus \pna(X)\neq \emptyset.$$

\begin{figure}[H]
\centering
\includegraphics[width=9.5cm]{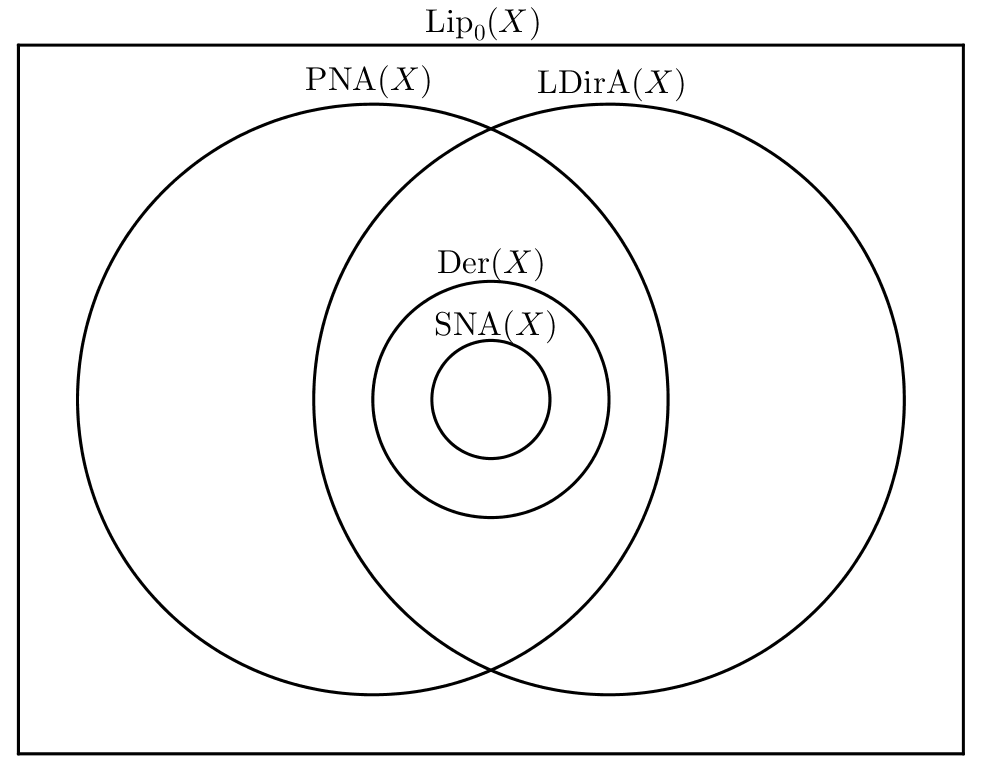}
\caption{Set inclusions for $\der$, $\ldira$, $\pna$ and $\lip$.}
\label{fig:venn-diagram}
\end{figure}

In what follows we will show that every region from the diagram in Figure \ref{fig:venn-diagram} is spaceable. To show this, we will only provide the arguments when the domain space is $\mathbb{R}$. In fact, as in \cite[Lemma 6.5]{CJLR2023}, if $X$ is any real Banach space and $Z \subseteq \lip(\mathbb{R})$ is an isometric copy of $c_0$ (resp. $\ell_\infty$) supported in $[0,+\infty)$, then the corresponding vector-valued version result can be obtained by considering the space $Z_2=\{\widetilde{f}\in\lip(X):\, f\in Z\}$, where $\widetilde{f}(x):=f(\|x\|)$ for all $x\in X$ and for each $f\in Z$. 

\begin{remark}\label{remark:der-spaceble}
It is clear that $\sna(X) \subseteq \der(X)$. Notice that $(\der(X)\setminus \sna(X))\cup\{0\}$ is clearly spaceable. In fact, $(\der(\mathbb{R})\setminus \sna(\mathbb{R}))\cup\{0\}$ contains an isometric copy of $c_0$ with basis functions $\{f_n\}_{n=1}^{\infty}$ given by $f_n(x):= \sin(x)$ if $x\in [4n\pi, (4n+2)\pi]$ and $f_n(x)=0$ otherwise (apply Lemmas \ref{lemma:equiv-sna_naf} and \ref{lemma:new-extension-lemma-[0,1]}). 
\end{remark}

\begin{proposition}\label{prop:lip-nopna-noldira-ell-infty}
For any real Banach space $X$, we have
\[
[\lip(X)\setminus (\pna(X)\cup \ldira(X)) ]\cup \{0\} \text{ isometrically contains $\ell_\infty$.}
\]
\end{proposition}

\begin{proof}
Let $\{I_n\}_{n=1}^\infty = \{[a_n,b_n]\}_{n=1}^\infty \subseteq \mathbb{R}$, where for each $n\in\bbn$, $a_n=2n-1$ and $b_n=2n$, and for simplicity, put $c_n = (a_n+b_n)/2=2n-1/2$. Let $\{s_n\}_{n=1}^\infty$ be the sequence of increasing prime numbers. For each $n \in \mathbb{N}$, we define $f_n \in \lip(\bbr)$ as follows:
\[
f_n (x): = 
\begin{cases}
(1-2^{-s_n^k}) (\frac{1}{2} -  d(x,c_{s_n^k}) ),\ &\text{ if } x \in I_{s_n^k} = [a_{s_n^k}, b_{s_n^k}] \text{ for some } k \in \mathbb{N}, \\ 
0,\ &\text{ else}.
\end{cases}
\]
Notice that $\|f_n\| = 1$ for every $n \in \mathbb{N}$ and that $\supp(f_n) \cap \supp(f_m) = \emptyset$ whenever $n \neq m$. 

Now, let $a=(a_n)_n \in \ell_\infty$ be given and consider $f^{(a)} = \sum_{n=1}^\infty a_n f_n$ (pointwise limit). Note that $\|f^{(a)}\|=\|a\|_\infty$ arguing for instance as in the proof of Lemma \ref{lemma:new-extension-lemma-[0,1]}. Similarly, arguing again as in Lemma \ref{lemma:new-extension-lemma-[0,1]}, it is routine to show that for any fixed $a \in \ell_\infty \setminus \{0\}$, $f^{(a)}$ neither attain its norm locally directionally nor attain its pointwise norm, since the local slopes and the pointwise slopes are always bounded by less than $\|a\|_\infty$ by construction.
\end{proof}

Before we present the next result, recall from \cite[Proposition 2.8]{CJLR2023} that there exists a Lipschitz function on $[0,1]$ which does not attain its pointwise norm. More precisely, given $0<\eps,\eta <1$, we can  construct $g_{\varepsilon,\eta} \in \lip ([0,1])$ as in Figure \ref{figure} below. Observe that $\|g_{\varepsilon,\eta} \| = 1$,  
\[
\sup_{q\in\bbr\setminus\{p\}} |S(g_{\varepsilon,\eta} ,p,q)| < 1 \text{ for any fixed } p \in (0,1], \,\, \text{ and } \sup_{q\in\bbr\setminus\{0\}} |S(g_{\varepsilon,\eta} ,0,q)| \leq \varepsilon.
\]
The graph of $g_{\varepsilon,\eta}$ is the union of certain cones $\{ [u_n,v_n] \cup [v_n, u_{n+1}] : n \in \mathbb{N}\}$ as in Figure \ref{figure}, and for each $n\in\bbn$, the segments $[u_n, v_n]$ and $[u_{n+1}, v_n]$ have respective slopes of $-(1-\eta^n)$ and $(1-\eta^n)$, $u_1=(1,0)$, and $u_n$ converges to $(0,0)$ as $n \rightarrow \infty$.
\vspace{-1em}

\begin{center}
\begin{figure}[H]
\centering
\includegraphics[width=10.5cm]{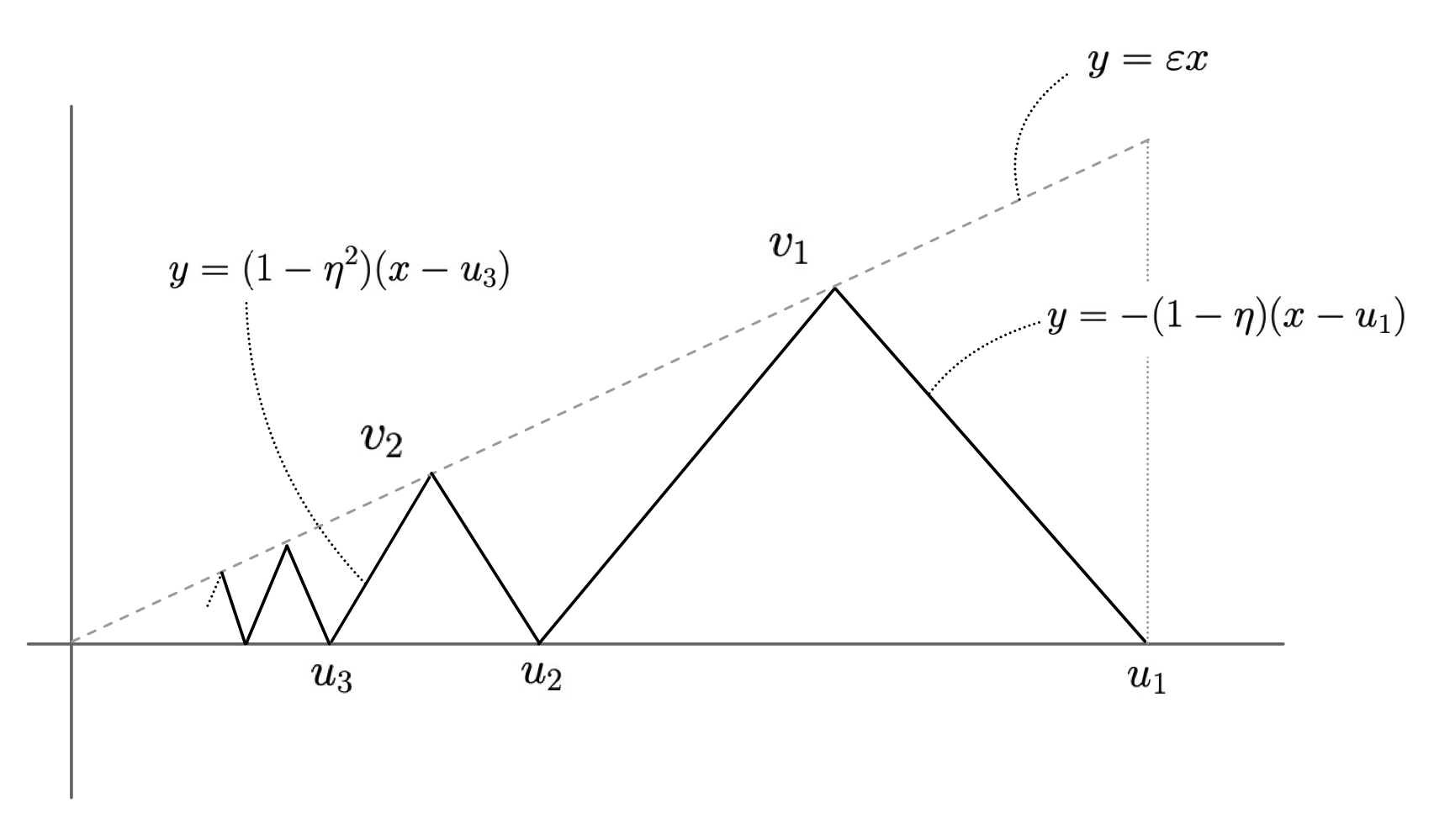}
\caption{The graph of the function $g_{\varepsilon,\eta} \in \Lip([0,1])$ }
\label{figure}
\end{figure}
\end{center}
\vspace{-2em}

\begin{proposition}\label{prop:ldira-nopna-ell-infty}
For any real Banach spaces $X$, we have
\[
(\ldira(X)\setminus \pna(X)) \cup \{0\} \text{ isometrically contains $\ell_\infty$.}
\]
\end{proposition}

\begin{proof}
We shall use the function $g = g_{\varepsilon,\eta} \in \lip ([0,1])$ in Figure \ref{figure} for some $0<\varepsilon,\eta<1$. 
Let $\{s_n\}_{n=1}^\infty$ be the sequence of increasing prime numbers. For each $n \in \mathbb{N}$, consider $h_n \in \lip(\bbr)$ whose graph is the union of line segments $[u_{s_n^m}, v_{s_n^m}]$ and $[v_{s_n^m}, u_{s_n^m +1}]$ for all $m \in \mathbb{N}$, and $x$-axis otherwise. Given $a=(a_n)_n \in \ell_\infty$, define $h^{(a)}$ to be the pointwise limit of $\sum_{n=1}^\infty a_n h_n$. Note that $\|h^{(a)}\|=\|a\|_\infty$ arguing as in Lemma \ref{lemma:new-extension-lemma-[0,1]}. Fix $a\in \ell_\infty\setminus \{0\}$. As the support of $h^{(a)}$ is contained in $[0,1]$, it is clear that that $h^{(a)} \in \ldira (\bbr)$ (see for instance the argument used in \cite[Proposition 6.7]{CJLR2023}). It remains to check that $h^{(a)} \not\in \pna(\bbr)$. First, note that if $p\notin [0,1]$, then 
\begin{align*}
\sup_{q\neq p} |S(h^{(a)}, p, q)| &\leq \max \left\{ \sup_{q\neq p} |S(h^{(a)}, 0, q)|,\, \sup_{q\neq p} |S(h^{(a)}, 1, q)| \right\}\\
&\leq \max\{\varepsilon \|a\|_\infty,\, (1-\eta)\|a\|_\infty\}<\|a\|_\infty.
\end{align*}
Moreover, if for each $k\in\bbn$ we denote $p_k$ to the first coordinate of $u_k$, it is clear that if $p\in [p_n, p_{n+1}]$ for some $n\in\bbn$, then 
$$\sup_{q\neq p} |S(h^{(a)}, p, q)| \leq (1-\eta^{n+1}) \|a\|_\infty <\|a\|_\infty.$$
Finally, by construction, as in \cite[Proposition 2.8]{CJLR2023},
$$\sup_{q\neq 0} |S(h^{(a)}, 0, q)| \leq \varepsilon \|a\|_\infty <\|a\|_\infty.$$
Thus, $h^{(a)}\notin \pna(\bbr)$. Therefore, $(\ldira(\bbr)\setminus \pna(\bbr)) \cup \{0\}$ contains an isometric copy of $\ell_\infty$.
\end{proof}

\begin{proposition}\label{prop:pna-noldira-ell-infty}
For any real Banach space $X$, we have
\[
(\pna(X)\setminus \ldira(X)) \cup \{0\} \text{ isometrically contains $\ell_\infty$.}
\]
\end{proposition}

\begin{proof}
Define for each $n \in \N$ the function $g_n\in\lip(\bbr)$ given by
$$
g_n(x) := \max \left\{ 0, \frac{2^{n^2} - 2^{(n-1)^2}}{2} - \left| x - \frac{2^{n^2} + 2^{(n-1)^2}}{2} \right| \right\} \quad \text{for }x\in\bbr.
$$
Let $\{s_n\}_n$ denote the sequence of increasing prime numbers, and for each $n\in\bbn$ define
$$
f_n(x) := \sum_{k=1}^\infty \left( 1 - \frac{1}{2^k} \right) g_{s_n^k} (x).
$$
Note that for all $(m,n)\in\widetilde{\mathbb N}$, $\supp(f_n)\cap \supp(f_m)=\emptyset$. For each $a = (a_n)_n \in \ell_\infty$, define $f^{(a)} = \sum_{n=1}^\infty a_n f_n $ by the pointwise limit. Note that $\|f^{(a)}\| = \|a\|$ for every $a \in \ell_\infty$ arguing as in Lemma \ref{lemma:new-extension-lemma-[0,1]}. Fix $a\in \ell_\infty\setminus \{0\}$. It is not difficult to check that $f^{(a)}\notin \ldira(\bbr)$, similar to with Proposition \ref{prop:lip-nopna-noldira-ell-infty}. On the other hand, we will show that $f^{(a)}\in\pna(M)$. If there exists some $n\in\bbn$ such that $|a_{n}|=\|a\|_\infty$, then we have
\begin{align*}
\sup_{q\neq 0} |S(f^{(a)}, 0, q)|&\geq \lim_{k\to\infty} \left|S\left(f^{(a)}, 0, \frac{2^{(s_n^k)^2} + 2^{(s_n^k -1)^2}}{2}\right)\right|\\
&=\lim_{k\to\infty} |a_n| \left( 1 - \frac{1}{2^k} \right) \frac{2^{(s_n^k)^2} - 2^{(s_n^k -1)^2}}{2^{(s_n^k)^2} + 2^{(s_n^k -1)^2}}=\|a\|_\infty.
\end{align*}
Otherwise, there exists a subsequence of $\bbn$, $\{n_k\}_{k=1}^{\infty}$, such that $|a_{n_k}|$ converges to $\|a\|_\infty$. In this case, we have
\begin{align*}
\sup_{q\neq 0} |S(f^{(a)}, 0, q)|&\geq \lim_{k\to\infty} \left|S\left(f^{(a)}, 0, \frac{2^{(s_{n_k}^k)^2} + 2^{(s_{n_k}^k -1)^2}}{2}\right)\right|\\
&=\lim_{k\to\infty} |a_{n_k}| \left( 1 - \frac{1}{2^k} \right) \frac{2^{(s_{n_k}^k)^2} - 2^{(s_{n_k}^k -1)^2}}{2^{(s_{n_k}^k)^2} + 2^{(s_{n_k}^k -1)^2}}=\|a\|_\infty.
\end{align*}
Therefore, $(\pna(\bbr)\setminus \ldira(\bbr)) \cup \{0\}$ contains an isometric copy of $\ell_\infty$.
\end{proof}

Finally, by carefully modifying a construction from \cite[Proposition 2.7]{CJLR2023}, we also have the following.

\begin{proposition}\label{prop:pna-ldira-noder-ell-infty}
For any real Banach space $X$, we have
\[
[(\pna(X)\cap \ldira(X))\setminus \der(X)] \cup\{0\} \text{ isometrically contains $\ell_\infty$.}
\]
\end{proposition}

\begin{proof}
It was shown in \cite[Propositions 2.8 and 6.6]{CJLR2023} that $\pna(\bbr)$ isometrically contains $\ell_\infty$ with canonical vectors isometrically given as the functions
$$f_n(x):=\max\left\{0,\, \frac{2^{2n-1}-1}{2^{n^2+1}} - \left| x-\frac{2^{2n-1}+1}{2^{n^2+1}} \right|\right\}\quad \text{for }x\in \bbr.$$
It was also implicitly shown that given $a=(a_n)_{n}\in\ell_\infty$, if $f^{(a)}$ denotes the pointwise limit of $\sum_n a_n f_n$, if there is some $n_0\in\bbn$ such that $|a_{n_0}|=\|a\|_\infty$, then $f^{(a)}$ attains its pointwise norm wherever $f_{n_0}$ does, and otherwise $f^{(a)}$ can only attain its pointwise norm at the point $x=0$. We apply Lemma \ref{lemma:special-c0-in-ell-infty} to this copy of $\ell_\infty$ to find a new space $Z$ isometrically isomorphic to $\ell_\infty$ and also isometrically contained in $\pna(\bbr)$ such that every function in $Z$ only attains its pointwise norm at $0$. Moreover, note that for all $f\in Z$, $\supp(f)$ is bounded, so $f\in\ldira(\bbr)$. Finally, let $f\in Z\setminus \{0\}$, and let $a=(a_n)_n\in\ell_\infty$ be the element such that $f=f^{(a)}$. If $f\in\der(\bbr)$, then the only point at which $f$ can attain its norm through derivatives is $x=0$, as that is the only point where $f$ attains its pointwise norm. However, the right-hand derivative of $f$ does not even exist at $0$. 
\end{proof}

\textbf{Acknowledgement}. \\
The authors are grateful to Sheldon Dantas, Vladimir Kadets, Daniel L. Rodr\'{\i}guez-Vidanes and Andres Quilis for fruitful conversations on the topic of the paper. Mingu Jung was supported by KIAS Individual Grants (MG086601, HP086601) at Korea Institute for Advanced Study and June E Huh Center for Mathematical Challenges.
Han Ju Lee and \'Oscar Rold\'an are supported by Basic Science Research Program, National Research Foundation of Korea (NRF), Ministry of Education, Science and Technology [NRF-2020R1A2C1A01010377]. \'Oscar Rold\'an is also supported by the Spanish project PID2021-122126NB-C33/MCIN/AEI/10.13039/501100011033 (FEDER).


\begin{thebibliography}{99}


\bibitem{AAAG07}
{M.~D.~Acosta, A.~Aizpuru, R.~M.~Aron, and F.~J.~Garc\'{\i}a-Pacheco}, 
\textit{Functionals that do not attain their norm}, 
{Bull. Belg. Math. Soc. Simon Stevin} \textbf{14} (2007), no. 3, 407--418.


\bibitem{ABPS}
{R.~M. Aron, L.~Bernal-Gonz\'alez, D.~M.~Pellegrino, and J.~B.~Seoane-Sep\'ulveda}, 
\textit{Lineability: the search for linearity in mathematics}, Monogr. Res. Notes Math., CRC Press, Boca Raton, FL, 2016.


\bibitem{AGS05} 
{R.~M. Aron, V.~I. Gurariy, and J. Seoane-Sep\'ulveda}, \textit{Lineability and spaceability of sets of functions on $\mathbb{R}$}, Proc. Amer. Math. Soc. \textbf{133} (2005), no. 3, 795--803. 

\bibitem{AMRT23} 
{A.~Avil\'es, G.~Mart\'{\i}nez-Cervantes, A.~Rueda~Zoca, and P.~Tradacete}, 
\textit{Infinite dimensional spaces in the set of strongly norm-attaining Lipschitz maps}. 
To appear in {Rev. Mat. Iberoam.}. 2023. 

    \bibitem{BG}  {P. Bandyopadhyay and G. Godefroy}, \textit{Linear structures in the set of norm attaining functionals on a Banach space}, J. Convex Anal. \textbf{13} (2006), no. 3--4, 489--497. 

\bibitem{CCGMR19} 
{B.~Cascales, R.~Chiclana, L.~C.~Garc\'{\i}a-Lirola, M.~Mart\'{\i}n, and A.~Rueda~Zoca}, 
\textit{On strongly norm attaining Lipschitz maps}, 
{J. Funct. Anal.} \textbf{277} (2019), no. 6, 1677--1717.

\bibitem{Chiclana22} 
{R.~Chiclana}, 
\textit{Cantor sets of low density and Lipschitz functions on $C^1$ curves}, 
{J. Math. Anal. Appl.} \textbf{516} (2022), no. 1, Paper No. 126489, 13 pp.

\bibitem{CGMR21} 
{R.~Chiclana, L.~C.~Garc\'{\i}a-Lirola, M.~Mart\'{\i}n, and A.~Rueda~Zoca}, 
\textit{Examples and applications of the density of strongly norm attaining Lipschitz maps}, 
{Rev. Mat. Iberoam.} \textbf{37} (2021), no. 5, 1917--1951.


\bibitem{Choi23} 
{G.~Choi}, 
\textit{Norm attaining Lipschitz maps toward vectors}, 
{Proc. Amer. Math. Soc.} \textbf{151} (2023), no. 4, 1729--1741. 

\bibitem{CCM20} 
{G.~Choi, Y.~S.~Choi, and M.~Mart\'{\i}n}, 
\textit{Emerging notions of norm attainment for Lipschitz maps between Banach spaces}, 
{J. Math. Anal. Appl.} \textbf{483} (2020), no. 1, Paper No. 123600, 24 pp.


\bibitem{CJLR2023} 
{G.~Choi, M. Jung, H. J. Lee, and \'O. Rold\'an}, 
\textit{Embeddings of infinite-dimensional spaces in the sets of norm-attaining Lipschitz functions}, preprint. https://arxiv.org/abs/2312.00393


\bibitem{CJ17} 
{M.~C\'uth and M.~Johanis}, 
\textit{Isometric embedding of $\ell_1$ into Lipschitz-free spaces and $\ell_\infty$ into their duals}, 
{Proc. Amer. Math. Soc.} \textbf{145} (2017), no. 8, 3409--3421.




\bibitem{DFJR23}
{S. Dantas, J. Falc\'o, M. Jung, and D. L. Rodr\'{\i}guez-Vidanes},
\textit{Linear structures in the set of non-norm-attaining operators on Banach spaces}, preprint. https://arxiv.org/abs/2311.17426 

\bibitem{DMQR23} 
{S.~Dantas, R.~Medina, A.~Quilis, and \'O.~Rold\'an}, 
\textit{On isometric embeddings into the set of strongly norm-attaining Lipschitz functions}, 
{Nonlinear Anal.} \textbf{232} (2023), Paper No. 113287, 15 pp.

\bibitem{EGS14}
{P. H. Enflo, V. I. Gurariy, and J. B. Seoane-Sep\'ulveda}, 
\textit{Some results and open questions on spaceability in function spaces}, Trans. Amer. Math. Soc. \textbf{366} (2014), no. 2, 611--625, DOI
10.1090/S0002-9947-2013-05747-9. 

\bibitem{FGMR}  {J. Falc\'o, D. Garc\'{\i}a, M. Maestre, and P. Rueda}, \textit{Spaceability in norm-attaining sets}, Banach J. Math. Anal. \textbf{11} (2017), no. 1, 90--107. 

\bibitem{FGK99}
{V. P. Fonf, V. I. Gurariy and M. I. Kadec}, 
\textit{An infinite dimensional subspace of $C[0, 1]$ consisting of nowhere differentiate functions}, 
C. R. Acad. Bulgare Sci., \textbf{52} (11-12) (1999), 13--16.






\bibitem{GPacheco08} 
{F.~J. Garc\'{\i}a-Pacheco}, 
\textit{Vector subspaces of the set of non-norm-attaining functionals}, 
{Bull. Aust. Math. Soc.}. \textbf{77} (2008), no. 3, 425--432.

\bibitem{GP2010}  {F.~J. Garc\'{\i}a-Pacheco and D. Puglisi}, \textit{Lineability of functionals and operators}, Studia Math. \textbf{201} (2010), 37--47.

\bibitem{GP2015}  {F.~J. Garc\'{\i}a-Pacheco and D. Puglisi}, \textit{A short note on the lineability of norm-attaining functionals in subspaces of $\ell_\infty$}, Arch. Math. \textbf{105}, (2015), 461--465.


\bibitem{Godard10}
{A.~Godard}, 
\textit{Tree metrics and their Lipschitz free spaces}, 
{Proc. Amer. Math. Soc.} \textbf{138} (2010), no. 12, 4311--4320.

\bibitem{Godefroy01} 
{G.~Godefroy}, 
\textit{The Banach space $c_0$}, 
{Extracta Math.} \textbf{16} (2001), no. 1, 1--25.

\bibitem{Godefroy15} 
{G.~Godefroy}, 
\textit{A survey on Lipschitz-free Banach spaces}, 
{Comment. Math.} \textbf{55} (2015), no. 2, 89--118.

\bibitem{Godefroy16} 
{G.~Godefroy}, 
\textit{On norm attaining Lipschitz maps between Banach spaces}, 
{Pure Appl. Funct. Anal.} \textbf{1} (2016), no. 1, 39--46.

\bibitem{Gurariy66} 
{V.~I. Gurariy}, 
\textit{Subspaces and bases in spaces of continuous functions} (Russian), Dokl. Akad. Nauk SSSR, \textbf{167} (1966), 971--973. 

\bibitem{GQ04} 
{V.~I. Gurariy and L. Quarta}, \textit{On lineability of sets of continuous functions}, J. Math. Anal. Appl. \textbf{294} (2004), no. 1, 62--72. 


\bibitem{IKW07} 
{Y.~Ivakhno, V.~Kadets, and  D.~Werner}, 
\textit{The Daugavet property for spaces of Lipschitz functions}, 
{Math. Scand.} \textbf{101} (2007), no. 2, 261--279.

\bibitem{James57} R.~C.~James, \textit{Reflexivity and the supremum of linear functionals}, Ann. of Math. \textbf{66} (1957), 159--169.

\bibitem{James64} R.~C.~James, \textit{Characterizations of reflexivity}, Studia Math. \textbf{23} (1964), 205--216.


\bibitem{JMR23} 
{M.~Jung, M.~Mart\'{\i}n, and A.~Rueda~Zoca}, 
\textit{Residuality in the set of norm attaining operators between Banach spaces}, 
{J. Funct. Anal.} \textbf{284} (2023), no. 2, Paper No. 109746, 46 pp.


\bibitem{KMS16} 
{V.~Kadets, M.~Mart\'{\i}n, and M.~Soloviova}, 
\textit{Norm-attaining Lipschitz functionals}, 
{Banach J. Math. Anal.} \textbf{10} (2016), no. 3, 621--637.

\bibitem{KR22} 
{V.~Kadets and \'O.~Rold\'an}, 
\textit{Closed linear spaces consisting of strongly norm attaining Lipschitz functionals}, 
{Rev. R. Acad. Cienc. Exactas F\'{\i}s. Nat. Ser. A Mat. RACSAM} \textbf{116} (2022), no. 4, Paper No. 162, 12 pp. 

\bibitem{LM}  {B. Levine, and D. Milman}, \textit{On linear sets in space $C$ consisting of functions of bounded variation}, Russian, with English summary. Commun. Inst. Sci. Math. M\'ec. Univ. Kharkoff [Zapiski Inst. Mat. Mech.] (4) {\bf 16}, 102--105 (1940).

\bibitem{L} {J.~Lindenstrauss}, \textit{On operators which attain their norm}, Israel. J. Math. \textbf{1} (1963), 139--148.

\bibitem{Matouskova2000} E. Matou\v{s}kov\'a, \textit{Extensions of continuous and Lipschitz functions}, Canad. Math. Bull. \textbf{43} (2000), 208--217.


\bibitem{PT}  {D. Pellegrino and E.V. Teixeira}, \textit{Norm optimization problem for linear operators in classical Banach spaces}, Bull. Braz. Math. Soc. \textbf{40} (3) (2009) 417--431.

\bibitem{PP}  {Ju. \={I}. Petun\={\i}n and A. N. Pl\={\i}\v{c}ko}, \textit{Some properties of the set of functionals that attain a supremum on the unit sphere}, Ukrain. Mat. \v{Z}. \textbf{26} (1974), 102–106, 143.






\bibitem{Read18} 
{C.~J.~Read}, 
\textit{Banach spaces with no proximinal subspaces of codimension 2}, 
{Israel J. Math.} \textbf{223} (2018), no. 1, 493--504. 

\bibitem{Rmoutil17} 
{M.~Rmoutil}, 
\textit{Norm-attaining functionals need not contain 2-dimensional subspaces}, 
{J. Funct. Anal.} \textbf{272} (2017), no. 3, 918--928.


\bibitem{Weaver18} 
{M.~Weaver}, 
\textit{Lipschitz Algebras}, 
2nd ed., World Scientific Publishing Co., River Edge, NJ, 2018.



\end{thebibliography}
\end{document}